\documentclass[10pt]{amsart}

\usepackage{color,soul}
\sethlcolor{yellow}
\usepackage[margin=1.25in]{geometry} 
\usepackage{kpfonts}
\usepackage{enumerate}
\usepackage[mathscr]{eucal}
\usepackage{color} 
\usepackage[colorlinks=true, linkcolor=red, citecolor=blue, menucolor=black]{hyperref}

\usepackage{amsmath,amsthm, amsfonts, amssymb,MnSymbol,extpfeil}
\numberwithin{equation}{section}
\usepackage{graphicx}
\usepackage{listings}
\usepackage{enumerate}
\usepackage{centernot}
\usepackage{enumitem}
\usepackage{xfrac}
\usepackage{leftindex}
\usepackage{hyperref}

\usepackage{amsmath, amsthm, amssymb}
\numberwithin{equation}{section}
\usepackage{xspace}
\usepackage{graphicx}
\usepackage{array}
\usepackage{braket}
\usepackage{geometry}
\usepackage{multicol}
\usepackage{mathtools}
\usepackage{enumerate}
\usepackage{delarray}
\usepackage{mathtools}
\usepackage{faktor,extpfeil} 
\usepackage{mathrsfs}
\usepackage{tikz-cd}
\usepackage{hyperref}
\usepackage[normalem]{ulem} 
\usepackage[italicdiff]{physics} 
\usepackage{bbm} 
\usepackage{float}
\usepackage{stmaryrd} 
\usepackage{aligned-overset}
\usepackage{xcolor}
\usepackage{cite}

\theoremstyle{plain}
\newtheorem{thm}{Theorem}[section]
\newtheorem{lemma}[thm]{Lemma}
\newtheorem{cor}[thm]{Corollary}

\newtheorem{maintheorem}{Theorem}

\newtheorem{maincor}[maintheorem]{Corollary}

\newtheorem{mainconj}[maintheorem]{Conjecture}
\newtheorem{prop}[thm]{Proposition}
\theoremstyle{remark}
\newtheorem*{claim}{Claim}
\theoremstyle{definition}

\newtheorem{defn}[thm]{Definition}
\newtheorem*{notation*}{Notation}

\newenvironment{subproof}[1][\proofname]{%
  \begin{proof}[#1]%
}{%
  \end{proof}%
}

\DeclarePairedDelimiterX{\inp}[2]{\langle}{\rangle}{#1, #2}

\newcommand{\sA}{\mathcal A}
\newcommand{\sB}{\mathcal B}
\newcommand{\sM}{\mathcal M}
\newcommand{\sN}{\mathcal N}
\newcommand{\sQ}{\mathcal Q}
\newcommand{\sR}{\mathcal R}

\makeatletter
\newcommand*{\bigcdot}{}
\DeclareRobustCommand*{\bigcdot}{%
  \mathbin{\mathpalette\bigcdot@{}}%
}
\newcommand*{\bigcdot@scalefactor}{.5}
\newcommand*{\bigcdot@widthfactor}{1.15}
\newcommand*{\bigcdot@}[2]{%
  \sbox0{$#1\vcenter{}$}
  \sbox2{$#1\cdot\m@th$}%
  \hbox to \bigcdot@widthfactor\wd2{%
    \hfil
    \raise\ht0\hbox{%
      \scalebox{\bigcdot@scalefactor}{%
        \lower\ht0\hbox{$#1\bullet\m@th$}%
      }%
    }%
    \hfil
  }%
}
\makeatother

\usepackage{enumitem}
\usepackage{color}
\newlist{steps}{enumerate}{1}
\usepackage[foot]{amsaddr}
\setlist[steps, 1]{label = Step \arabic*:}

\hypersetup{%
  colorlinks=true,%
  linkcolor=blue,%
  citecolor=blue,%
  filecolor=blue,%
  menucolor=blue,%
  urlcolor=blue,%
  pdfnewwindow=true,%
  pdfstartview=FitBH
}   

\title{pp}

\newcommand{\HH}{{\mathcal H}}

\newcommand{\sI}{\mathscr I}

\newcommand{\ur}{{\textbf{r}}}
\newcommand{\sP}{\mathcal P}

\begin{document}
\title{Relative Biexactness for Relative Hyperbolic Groups and Some Applications}

\author{Ionu\c{t} Chifan}
\address{The University of Iowa, Mathematics Department, 14 MacLean Hall, Iowa City,
IA, 52240, USA}
\author{Kai Toyosawa}
\address{University of Münster, Mathematics Münster, Einsteinstrasse 62
48149 Münster, Germany}
\author{Zhiyuan Yang}
\address{Purdue University, Department of mathematics,  150 N University St, West Lafayette, IN, 47907, USA}
\maketitle

\begin{abstract} In this paper, we confirm a conjecture of Ozawa and others asserting that every finitely generated, relatively hyperbolic, exact group is bi-exact (in the sense of Ozawa) relative to its natural peripheral structure. As a consequence, every such group gives rise to a prime group von Neumann algebra.

As an application, we construct a continuum family of property (T), relatively hyperbolic groups $\{G_i\}_{i\in I}$ such that, for every fixed arbitrary free, ergodic, probability measure-preserving action $G_i \curvearrowright Z_i$, the collection of associated group measure space von Neumann algebras $\{L^\infty(Z_i)\rtimes G_i\}_{i\in I}$ are pairwise non-stably $\ast$-isomorphic.   
    
\end{abstract}

\section{Introduction}

Bi-exactness was introduced by Ozawa in his seminal work \cite{Oz03} as a new C$^*$-algebraic approximation property for discrete groups. One of its fundamental features is that it has strong structural consequences for the associated group von Neumann algebras. Indeed, Ozawa proved that the group von Neumann algebra of every bi-exact group is solid, meaning that the relative commutant of every diffuse von Neumann subalgebra is amenable. He further proved that Gromov hyperbolic groups are bi-exact, thereby establishing solidity of their group von Neumann algebras \cite{Oz03}.

Motivated by these developments, Ozawa introduced the notion of \emph{relative bi-exactness} \cite{Oz06} for a group $G$ relative to a family of subgroups $\mathcal P=\{H_\lambda \}_{\lambda}$. Roughly speaking, $G$ is bi-exact relative to $\mathcal P$ if it admits a compactification $X\supseteq G$ such that the left translation action of $G$ extends continuously to $X$, the right translation action extends continuously to $X\setminus\partial G$, where $\partial G=X\setminus G$, and the induced action of $G\times G$ on $\partial G$ is topologically amenable relative to the family $\mathcal P$ (see Definition \ref{relbiexactdef}). This notion extends ordinary bi-exactness and retains its fundamental operator algebraic consequences. Motivated by these applications, Ozawa and, independently, several other authors conjectured that every group which is hyperbolic relative to exact peripheral subgroups should be bi-exact relative to its peripheral structure.

\begin{mainconj}\label{relbiexactconj}
Let $G$ be a group that is hyperbolic relative to a family of exact subgroups $\mathcal P=\{H_\lambda\}_\lambda$. Then $G$ is bi-exact relative to $\mathcal P$.
\end{mainconj}

The importance of Conjecture~\ref{relbiexactconj} lies in the fact that, once established, general von Neumann algebra theory immediately yields the following relative solidity result.

\begin{maincor}[Relative solidity]\label{main: relsol} Let $G$ be a finitely generated, non-elementary group that is hyperbolic relative to exact proper subgroups $\mathcal P=\{ H_\lambda\}_\lambda.$  Let $G \curvearrowright \mathcal N$ be a trace preserving action on an amenable von Neumann algebra and denote by $\mathcal{M}=\mathcal N \rtimes G$ be the corresponding cross-product von Neumann algebra.    Let $p \in \sM$ be a nonzero projection and let $\mathcal{A} \subseteq p\mathcal{M}p$ be a von Neumann subalgebra which has no corner that interwtines into $\mathcal N$ inside $\sM$  the sense of Popa (see Subsection \ref{popaint}).

Then one can find orthogonal central projections $p_0,  \dots, p_n \in \mathscr{Z}(\mathcal{A} \vee (\sA'\cap  p\sM p))$ with $p_0 + \dots + p_n = p$ that are maximal with the following properties: 
    \begin{enumerate}
        \item $(\mathcal{A}^{'} \cap p\mathcal{M}p)p_0$ is amenable.
        \item Any corner of $(\mathcal{A} \vee (\sA'\cap p\sM p))p_i$ has a nontrivial subcorner which can be unitarily conjugated into $\mathcal N \rtimes H_i$, for every $i=\overline{1,n}$.
\end{enumerate}
    
\end{maincor}

The localization of commuting subalgebras afforded by relative bi-exactness has become one of the fundamental tools in the deformation/rigidity theory of von Neumann algebras. Indeed, this localization principle has become a key ingredient in many classification and rigidity results for von Neumann algebras arising from countable groups and their probability measure preserving actions. Motivated by the central role of boundary pieces and bi-exactness in these developments, the study of this notion has intensified in recent years; see \cite{DKE22,DKEP23,DP23,DD25}, just to mention a few. Of particular significance are the extensions of boundary pieces and the small-at-infinity boundary from groups to general von Neumann algebras developed by Ding, Kunnawalkam Elayavalli, and Peterson \cite{DKEP23}, together with the corresponding notion of von Neumann algebraic bi-exactness introduced by Ding and Peterson \cite{DP23}. These developments established several new structural results, most notably that bi-exactness is a von Neumann algebra invariant \cite{DP23}.

\vskip 0.08in 
At present, the conjecture has been verified for numerous important classes of relatively hyperbolic groups. Ozawa first established it for free products in \cite{Oz06}. We note that for the more general class of non-necessarily exact free product groups, the corresponding relative solidity result had previously been proved in \cite{CH10}, building on Popa's  deformation/rigidity techniques introduced in the influential work of Ioana, Peterson, and Popa \cite{IPP05}. More recently, Oyakawa constructed arrays (see Definition \ref{defn: array}) on relatively hyperbolic groups using deep geometric group theoretic methods \cite{MY09,HO13}. His construction consists of two
complementary parts. The first array controls the geometry of the coned-off
Cayley graph, while the second uses the separating cosets of Hull and Osin,
together with proper arrays on the peripheral subgroups. This enabled him to establish the conjecture when the peripheral subgroups are bi-exact \cite{Oya23}. Furthermore, these arrays were recently combined with deformation/rigidity techniques and group-theoretic Dehn filling to prove the relative solidity result of Theorem~\ref{main: relsol} in the case where the peripheral subgroups are residually finite \cite{AMAKCK}.
\vskip 0.08in

A second line of development closely related to our approach concerns {upgrading techniques} for properties relative to boundary pieces. Roughly speaking, such arguments start with a property relative to a larger, more tractable boundary piece and then upgrade it to a smaller boundary piece by analyzing the difference between the two. This perspective was first systematically developed by Ding and Kunnawalkam Elayavalli in the context of proper proximality \cite{DKE22}, and has since played an important role in that theory; see, for instance, \cite{DKE21,TY25,Din25,KE26}. More recently, related upgrading techniques have also been developed for relative bi-exactness \cite{TY25,KEY26}. Our main proof can be viewed as another instance of this approach. One difference, however, is that the larger boundary pieces in our argument are not generated by subgroups. Instead, they arise naturally from products of peripheral cosets, and we obtain the desired upgrade by analyzing the successive quotients in the resulting filtration.

\vskip 0.08in
In this paper we make new progress on Conjecture \ref{relbiexactconj} by confirming it for a wide class of relative hyperbolic groups.

 \begin{maintheorem}\label{thm: main2} Conjecture \ref{relbiexactconj} holds true whenever $G$ is finitely generated and $\mathcal P$ is finite.
     
 \end{maintheorem}

Our proof brings together these two lines of development. More precisely, we combine the arrays on relatively hyperbolic groups constructed by Oyakawa \cite{Oya23} with a boundary-piece upgrading argument based on a new decomposition (Theorems \ref{thm: decomposition of Jk}, \ref{orthproj}, and \ref{thm: hyperbolic embedding amenable action}).  A key geometric ingredient in our argument is the separation property of peripheral cosets. This same property also plays a central role in Oyakawa's construction of his second array. More precisely, we introduce an increasing family of boundary piece ideals \((J_k)_{k\geq 1}\), such that Oyakawa's first array \cite{Oya23} yields the relative biexactness with respect to the union boundary piece ideal $ J_\infty:= \overline{\bigcup_n J_n} $. We then decompose the spectrum of the successive quotients $J_k/J_{k-1}$ into compact pieces whose stabilizers are built from conjugates of the peripheral subgroups. Using exactness of the peripheral subgroups, we obtain amenability of these stabilizer actions and glue them along the filtration to obtain the desired boundary amenability. This decomposition of $J_k/J_{k-1}$ arises naturally from the geometry
of hyperbolically embedded subgroups \cite{DGO17}. Structurally, the proof is similar to the boundary-piece upgrading argument of \cite{KEY26}, where amenability of the quotient boundary action is obtained by decomposing an appropriate quotient into stabilizer pieces and then gluing up their amenable actions.

Thus, Theorem \ref{thm: main2} implies the relative solidity result established in Theorem \ref{main: relsol}, which in turn yields the following primeness theorem for the von Neumann algebras associated with relatively hyperbolic groups.

\begin{maincor}
Let $G$ be any non-elementary, finitely generated group which is hyperbolic relative to a finite family of exact, proper subgroups $\mathcal P=\{H_\lambda\}_\lambda$. Then its group von Neumann algebra $L(G)$ is s-prime. That is, for any nonzero projection $p\in L(G)$, there do not exist diffuse, commuting von Neumann subalgebras $\mathcal M,\mathcal N\subset pL(G)p$ whose generated von Neumann algebra $\mathcal M\vee\mathcal N\subseteq pL(G)p$ has finite index.
\end{maincor}

In particular, for every such group, $L(G)$ cannot be decomposed as a tensor product of diffuse von Neumann algebras; that is, it is prime. This result unifies and extends several previous primeness results in the literature, including \cite{Ge98}, \cite{Oz03}, \cite[Theorem A]{CKP14}, and \cite[Theorem D]{AMAKCK}. It also provides new evidence for the broader and more difficult conjecture that the group von Neumann algebra of every non-elementary acylindrically hyperbolic group is prime; see \cite[Conjecture 8.2]{CKP14}.

\vskip 0.08in

Our main relative solidity result, Theorem~\ref{main: relsol}, can also be used to derive a new rigidity result for crossed product von Neumann algebras. Specifically, by combining it with techniques from the theory of unique Cartan subalgebras \cite{IPP05,AMAKCK}, measure equivalence \cite{Sak09}, and structural constructions in geometric group theory \cite{Ga00,AMO06}, we construct a continuum of property~(T) groups satisfying a group measure space version of Connes Rigidity Conjecture.

\begin{maintheorem}\label{mainresult:nonW^*e}
There exists a continuum of property~(T) relatively hyperbolic groups
$\{G_\imath\}_{\imath\in\sI}$ with the following property. For each
$\imath\in\sI$, let $G_\imath\curvearrowright Z_\imath$ be an arbitrary free, ergodic,
pmp action on a standard probability space. Then the associated group measure space von Neumann algebras
$\{L^\infty(Z_\imath)\rtimes G_\imath\}_{\imath\in\sI}$
are pairwise non stably isomorphic.
\end{maintheorem}

Over the years, several continuum-sized families of countable groups that are pairwise non-measure equivalent have been constructed. For example, Gaboriau's work on cost and $L^2$-Betti numbers \cite{Ga00,Ga02} yields continuum many pairwise non-measure equivalent countable groups. More recently, Drimbe and Vaes proved in \cite[Theorem A]{DV25} the existence of a continuum of torsion-free, non-finitely generated countable groups that are pairwise not von Neumann equivalent, and hence, in particular, pairwise non-measure equivalent. Furthermore, Ioana and Tucker-Drob showed in \cite[Theorem A]{ITD} that a continuum family of property~(T), torsion-free wreath-like product groups introduced in \cite{CIOS2} are pairwise non-measure equivalent. Independently, using completely different methods based on $\ell^2$-Betti numbers, Fournier-Facio and Sun established in \cite[Theorem D]{FFS} the existence of a continuum of pairwise non-measure equivalent finitely generated, torsion-free groups, which may moreover be chosen to have property~(T).
\vskip 0.08in
In particular, Theorem~\ref{mainresult:nonW^*e} provides a new continuum of pairwise non-measure equivalent property~(T) groups, thereby adding examples of a different flavor to the families constructed in the aforementioned works.

\begin{maincor}
There exists a continuum of property~(T) relatively hyperbolic groups $\{G_\imath\}_{\imath\in\mathscr I}$ such that $G_\imath \not\cong_{\mathrm{ME}} G_\jmath$ for all distinct $\iota,\jmath\in\mathscr I$.
\end{maincor}

 \noindent{\bf Acknowledgments.} The authors thank Professors Daniel Drimbe, Adrian Ioana, Srivatsav Kunnawalkam Elayavalli, Koichi Oyakawa, and Jesse Peterson for their helpful comments and suggestions.
 
 The first author was partially supported by the NSF grants DMS-2452247 and DMS-. The second author was funded by the Deutsche Forschungsgemeinschaft (DFG, German Research Foundation) under Germany's Excellence Strategy EXC 2044/2 –390685587, Mathematics M\"unster: Dynamics–Geometry–Structure.
 
 \noindent{\bf AI tool disclosure.} ChatGPT 5.5 and 5.6 Sol was used for English-language editing, proofreading, grammatical corrections, and improving the organization and clarity of expository material. In the early stages of the project, it was also used to assist with literature searches for related theorems and references. All mathematical results, including the formulation and proofs of the theorems, were developed solely by the human authors.

\section{Preliminaries on Hyperbolically Embedded Subgroups}

 In this section, we recall the notion of hyperbolically embedded subgroups together with several fundamental properties that will be used throughout the paper.

Let $G$ be a discrete group, let $X \subset G$ be a subset, not necessarily a generating set, and let $\mathcal P \coloneqq \{H_\lambda\}_{\lambda\in\Lambda}$ be a family of subgroups of $G$. Define $\HH\coloneqq\bigsqcup_\lambda(H_\lambda\setminus\{1\})$, where the union is understood as a formal disjoint union rather than as a subset of $G$. The relative Cayley graph $\Gamma(G,X\sqcup\HH)$ associated with $(G,X,\{H_\lambda\}_{\lambda\in\Lambda})$ is the Cayley graph of $G$ with respect to the alphabet $X\sqcup\HH$. Thus, its vertex set is $G$, and its set of oriented edges is $G\times (X\sqcup\HH)$, where an edge $(g,s)$ has initial vertex $g$ and terminal vertex $gs$.

For each $\lambda \in \Lambda$, the subgraph $\Gamma(H_\lambda, H_\lambda\setminus\{1\}) \subset \Gamma(G,X\sqcup\HH)$ is a complete graph and therefore has diameter one. To retain information about the intrinsic geometry of $H_\lambda$, one introduces the relative metric $\widehat d_\lambda\colon H_\lambda\times H_\lambda\rightarrow[0,\infty]$. A path in $\Gamma(G,X\sqcup\HH)$ is called \emph{$H_\lambda$-admissible} if it contains no edge of the distinguished subgraph $\Gamma(H_\lambda, H_\lambda\setminus\{1\})$. For $f,g\in H_\lambda$, the relative distance $\widehat d_\lambda(f,g)$ is defined to be the length of a shortest $H_\lambda$-admissible path from $f$ to $g$. If no such path exists, one sets $ \widehat{d}_{\lambda}(f,g)=\infty$.
We also extend $ \widehat{d}_{\lambda} $ to $ \widehat{d}_{\lambda}: G\times G\to [0,\infty] $ by defining $ \widehat{d}_{\lambda}(f,g)=\widehat{d}_{\lambda}(1,f^{-1}g) $ if $ f^{-1}g\in H_\lambda $ and $ \widehat{d}_{\lambda}(f,g)=\infty $ otherwise.

\begin{defn}
Let $G$ be a group and let $X\subseteq G$. A family of subgroups $\{H_\lambda\}_{\lambda\in\Lambda}$ is said to be \emph{hyperbolically embedded in $G$ with respect to $X$}, abbreviated as $\{H_\lambda\}_{\lambda\in\Lambda}\hookrightarrow_h(G,X)$ throughout the paper, if the following conditions hold:
\begin{enumerate}
\item $G= \langle X,\HH\rangle$.
\item The relative Cayley graph $\Gamma(G,X\sqcup\HH)$ is hyperbolic.
\item For every $\lambda\in\Lambda$, the metric space $(H_\lambda,\widehat d_\lambda)$ is proper; that is, every ball of finite radius contains only finitely many elements.
\end{enumerate}
\end{defn}

It was proved in \cite[Proposition 4.28]{DGO17} (see also \cite[Proposition 2.16]{Oya23}) that if $G$ is hyperbolic relative to $\{H_\lambda\}_{\lambda\in\Lambda}$, then $\{H_\lambda\}_{\lambda\in\Lambda}\hookrightarrow_h (G,X)$ for some, equivalently any, finite relative generating set $X$ of $G$. In fact, when $ \Lambda $ is finite, the converse implication also holds.

Throughout the paper, the subgroups in $\mathcal P=\{H_\lambda\}_{\lambda\in \Lambda}$ will be referred to as \emph{the peripheral subgroups} of $G$. A fundamental property of these subgroups, which plays a crucial role in our subsequent arguments, is their mutually almost malnormality. 

\begin{prop}[{\cite[Proposition 4.33]{DGO17}}]\label{Prop:maln}
Let $G$ be a group together with a family of hyperbolically embedded subgroups,   $ \{H_\lambda\}_{\lambda \in \Lambda} \hookrightarrow_{h} (G,X) $. Then for every $\lambda\in \Lambda$ and $g\in G\setminus H_\lambda$, we have $|H_\lambda \cap g^{-1}H_\lambda g|<\infty $. Also, if $\lambda\ne \mu$, then $|H_\mu \cap g^{-1}H_\lambda g|<\infty$ for all $g\in G$.
\end{prop}

\vskip 0.08in 

We now recall the notion of separating cosets along with their main properties. Our exposition follows \cite{HO13,Oya23} and the notations within.

\begin{defn}\cite[Definition 4.1]{Osi16}\cite[Definition 2.17]{Oya23}
    Let $ p $ be a fixed path in $ \Gamma(G,X\sqcup \HH) $. A nontrivial subpath $q$ of $p$ is called an \emph{$H_\lambda$-subpath} if every edge of $q$ is labeled by an element in $H_\lambda\setminus \{1\}$. A maximal $H_\lambda$-subpath of $p$ is called an \emph{$H_\lambda$-component of $p$}. 
    
    Two $H_\lambda$-components $q_1$ and $q_2$ of $p$ are said to be \emph{connected} if their vertices lie in the same left coset of $H_\lambda$. An $H_\lambda$-component $q$ of $p$ is called \emph{isolated} if it is not connected to any other $H_\lambda$-component of $p$.

    In the case that $p$ is a closed path, the definition above allows $q$ to be a subpath of any cyclic shift of $p$.
\end{defn}

\begin{prop}[{\cite[Proposition 4.13]{DGO17}\cite[Lemma 2.4]{HO13}\cite[Proposition 2.19]{Oya23}}]\label{prop: uniform bound for isolated components}
    Suppose $\{H_\lambda\}_{\lambda \in \Lambda} \hookrightarrow_{h} (G,X)$. There exists a constant $C$ such that for every geodesic $n$-gon $p$ in $\Gamma(G,X\sqcup \HH)$ and any isolated $H_\lambda$-component $a$ of $p$, we have
    \[ \widehat{d}_\lambda(a_-,a_+)\leq nC, \]
    where $a_- $ is the beginning of $a$ and $a_+$ is the end of $a$.
\end{prop}

A first important consequence of the preceding $n$-gon inequality is the following combinatorial property, often referred to as \emph{asymptotic freeness}, which generalizes the almost malnormality of the hyperbolically embedded subgroups from Proposition \ref{Prop:maln}. This property has been proved to be a powerful tool in establishing structural results for operator algebras; e.g., see the recent work \cite{CDS23,CIOS3,AMAKCK,AMCFQI26}. As in the prior works \cite{CDS23,AMAKCK}, it will also play a key role in our proof of the unique Cartan subalgebra results in Section \ref{sec: W*-rigidity}.

\begin{thm}[{\cite[Theorem 2.8]{CDS23}}]\label{evenlength} Let $G$ be a group with a hyperbolically embedded subgroup $H$. Then, for every finite set $K\subset G\setminus H$, there exists a finite set $L\subset H$ such that for all $m,n\geq 1$,
\begin{equation*}
    ((H\setminus L)K)^{m}\cap (K(H\setminus L))^{n}=\emptyset.
\end{equation*}
Here, $(S_1S_2)^{m}=S_1S_2S_1S_2\cdots S_1S_2$ denotes the product obtained by repeating $S_1S_2$ exactly $m$ times.
\end{thm}

This can be rephrased in the following more convenient way to work with later on.

\begin{cor}[{\cite[Corollary 2.7]{AMCFQI26}}]\label{nontrivial} Let $G$ be a group with a hyperbolically embedded subgroup $H$. Then for every finite set $K\subset G\setminus H$, there exists a finite set $L\subset H$ such that the following holds: for every $\ell\geq 1$ and every $g_1,\dots,g_\ell\in K, h_1,\dots,h_\ell\in H\setminus L$, we have
$g_1h_1\cdots g_\ell h_\ell\not=1$.
\end{cor}

We conclude this section with a second, more involved, application of the $n$-gon inequality. Specifically, we establish a separation property for subsets naturally associated with the peripheral subgroups. This result will play a key role in the proof of our main relative bi-exactness theorem in Section~\ref{sec: relhyperbolicimplies biexact}. To state it, we first recall the following definition.
 \begin{defn}
    Let $ \mathcal{S} $ be a family of subsets of $ G $. A subset $ T\subset G $ is called  \emph{small relative to $ \mathcal{S}$} or \emph{$\mathcal S$-small} if  $T \subset \bigcup_{i=1}^k s_iS_it_i$, for some $ s_i,t_i\in G $ and $S_i\in \mathcal{S}$. In particular, if the left and right $G$ translations permute $ \mathcal{S}$, then $T$ is small relative to $\mathcal{S}$ if $ T\subset \cup_{i=1}^k S_i $, for some $ S_i\in \mathcal{S} $.
\end{defn}

\begin{thm}\label{thm: decomposition of Jk} Let $G$ be a group together with $\{H_\lambda\}_{\lambda\in \Lambda} \hookrightarrow_h (G,X)$. 
    For every integer $k\geq 0$ consider the following collection of subsets of $G$,
    \[\mathbb A_k \coloneqq\{ a_0 H_{\lambda_1} a_1 H_{\lambda_2}a_2\cdots a_{k-1}H_{\lambda_k}a_k \,:\, a_i\in G, \lambda_i\in \Lambda \},\]
    where when $k=0$, we set $ \mathbb A_0 \coloneqq \{\{g\}:g\in G\} $.
    
    Fix $k\geq 1$ and let $T \coloneqq a_0H_{\lambda_1}a_1\cdots a_{k-1}H_{\lambda_k}a_k, \quad S\coloneqq b_0H_{\mu_1}b_1\cdots b_{k-1}H_{\mu_k}b_k\in \mathbb A_{k}$, where $a_i,b_i\in G$. Then either $ T\cap S $ is small relative to $\mathbb A_{k-1}$ or $ T=S $.
\end{thm}

\begin{proof}
We may assume $b_0=b_k=1$. Moreover, whenever
$a_i\in H_{\lambda_i}$ or $a_i\in H_{\lambda_{i+1}}$, we replace $a_i$
by $1$, and similarly for the $b_i$'s. Suppose that $ T\cap S $ is not $\mathbb A_{k-1}$-small, we will show that $T=S$.

For the two fixed subsets $T,S\in \mathbb A_k$ only, we enlarge the generating set $X$ so that it contains all nontrivial coefficients
$a_i,b_j^{-1}$, since adding finitely many elements to $X$ does not change the hyperbolic embeddedness. Let $\widehat d_\lambda$ denote the associated relative metric
on $H_\lambda$.

Fix $R\ge 6kC$, where $C$ is the constant in Proposition
\ref{prop: uniform bound for isolated components} for our $X$. For each $\lambda\in\Lambda$, put
\[
B_{\lambda,R}
 \coloneqq\{h\in H_\lambda:\widehat d_\lambda(1,h)\leq R\},
\]
and $ B^c_{\lambda,R}= H_\lambda\backslash B_{\lambda,R} $. Define
\[
T_R=a_0B_{\lambda_1,R}^ca_1\cdots a_{k-1}B_{\lambda_k,R}^ca_k,
\qquad
S_R=B_{\mu_1,R}^cb_1\cdots b_{k-1}B_{\mu_k,R}^c.
\]
Since each $B_{\lambda,R}$ is finite, 
$T\setminus T_R$ and $S\setminus S_R$ are $\mathbb A_{k-1}$-small.
Hence $T_R\cap S_R$ is still not $\mathbb A_{k-1}$-small.

Let $g\in T_R\cap S_R$. Choose decompositions
\[
g=a_0h_1a_1\cdots a_{k-1}h_ka_k
=r_1b_1\cdots b_{k-1}r_k,
\]
with $\widehat d_{\lambda_i}(1,h_i)>R$ and
$\widehat d_{\mu_i}(1,r_i)>R$.

These two expressions determine a geodesic $N$-gon, $N\le 4k$,
whose boundary label is
\[
a_0h_1a_1\cdots a_{k-1}h_ka_k
r_k^{-1}b_{k-1}^{-1}\cdots b_1^{-1}r_1^{-1}.
\]
More precisely, the edges for the geodesic $N$-gon are labeled by $ h_i\in H_{\lambda_i}, r_i^{-1}\in H_{\mu_i}$, and those nontrivial $ a_i,b^{-1}_i \in X$.

For each $i$, let
\[
P_i=a_0h_1a_1\cdots h_{i-1}a_{i-1},
\qquad
Q_i=r_1b_1\cdots r_{i-1}b_{i-1}.
\]
Then the edges labeled by $h_i$ and $r_i^{-1}$ are the
$H_{\lambda_i}$- and $H_{\mu_i}$-components
$[P_i,P_ih_i]$ and $[Q_ir_i,Q_i]$.

Since $N\le 4k$, Proposition
\ref{prop: uniform bound for isolated components} implies that every
isolated component has relative length at most $4kC$. Since $\widehat d_{\lambda_i}(1,h_i)>R\ge 6kC$,
and similarly for the $r_i$'s, none of these components can be isolated.

\medskip
\noindent\textbf{Claim.}
For fixed $T $ and $S$, there exists an $\mathbb A_{k-1}$-small subset $B$, such that if some component $[P_i,P_ih_i]$ is connected to a component other than
$[Q_ir_i,Q_i]$, then $g\in B$.

\begin{subproof}[Proof of the Claim]

Let
\[
B\coloneqq
\bigcup_{c_0,\ldots,c_s}
c_0H_{l_1}\cdots H_{l_s}c_s,
\]
where the union is taken over all integers $s<k$, all elements $c_0,\ldots,c_s\in
\{1,a_0,\ldots,a_k,b_1,\ldots,b_{k-1}\}$,
and each subgroup $H_{l_i}$ is one of $H_{\lambda_1},\ldots,H_{\lambda_k},
H_{\mu_1},\ldots,H_{\mu_k}$.
By construction, $B$ is $\mathbb A_{k-1}$-small.

Suppose that the component $[P_i,P_ih_i]$ is connected to a component
other than $[Q_ir_i,Q_i]$. We distinguish two possibilities which we analyze separately. Note that since $ a_i,b_j\in X\sqcup \{1\} $ by assumption, they contain no component by definition.

\smallskip
\noindent\textbf{Case (1):}
$P_iH_{\lambda_i}=P_jH_{\lambda_j}$ for some $j\neq i$.

\smallskip
\noindent\textbf{Case (2):}
$P_iH_{\lambda_i}=Q_jH_{\mu_j}$ for some $j\neq i$.

Without loss of generality assume $i<j$.
Suppose we are in Case (2). Then $Q_j=P_i\eta$ for some
$\eta\in H_{\lambda_i}$. Therefore
\[
g
 =Q_jr_jb_j\cdots b_{k-1}r_k
 =P_i\eta r_jb_j\cdots b_{k-1}r_k
 \in
a_0H_{\lambda_1}\cdots
a_{i-1}H_{\lambda_i}
b_j\cdots
H_{\mu_k}
\subset B.
\]

\noindent Similarly, in Case (1) we have 
$g\in
a_0H\cdots a_{i-1}Ha_j\cdots Ha_{k-1}Ha_k
\subset B$.
\end{subproof}


Since $B$ is $\mathbb A_{k-1}$-small while
$T_R\cap S_R$ is not, we may choose $g\in (T_R\cap S_R)\setminus B$. For this element, every $h_i$-component is connected to its corresponding $r_i^{-1}$-component by the claim. Hence
\[
P_iH_{\lambda_i}=Q_iH_{\mu_i},
\qquad
\lambda_i=\mu_i,\quad \text{ for all }i.\]

For $i=1$, we have $ P_1H_{\lambda_1}=Q_1H_{\lambda_1} $. Choose an element $ \eta_1\in H_{\lambda_1} $ such that $ Q_1=P_1\eta_1 $, we have $a_0\eta_1=1$ and thus $a_0=\eta^{-1}_1\in H_{\lambda_1}$. For $i\ge 2$, choose $\eta_i\in H_{\lambda_i}$ such that
$Q_i=P_i\eta_i$. Since
$P_{i+1}=P_ih_ia_i$, and 
$Q_{i+1}=Q_ir_ib_i$, we obtain $h_ia_i\eta_{i+1}
=\eta_ir_ib_i$, which implies $b_i\in H_{\lambda_i}a_iH_{\lambda_{i+1}}$ for $i\geq 1$. Finally, we have $Q_k=P_k\eta_k=gr_k^{-1}=ga_k^{-1}h_k^{-1}\eta_k$,
which implies $a_k=h_k^{-1}\eta_k r_k\in H_{\lambda_k}$.

In conclusion, all coefficients differ only by multiplication by adjacent
peripheral subgroup elements, which does not change the corresponding
product set. Therefore, $T=S$.\end{proof}

Note that when $ k=1 $, the above condition is precisely the mutual almost malnormality in Proposition \ref{Prop:maln}. Therefore, the above theorem can also be considered as a strengthening of almost malnormality.

\section{Relative Biexactness of Relative Hyperbolic Groups}\label{sec: relhyperbolicimplies biexact}

In this section we establish our main relative biexactness result. We start by recalling the notion of relative biexactness.

\begin{defn}
    A \emph{boundary piece $I$} of $G$ is a left and right $G$-invariant $C^*$-subalgebra of $ \ell^\infty(G) $.
    If $ \mathcal{S} $ is a family of subsets in $ G $, we denote $ I_{\mathcal{S}} $ the smallest boundary piece that contains $\ell^\infty(A)\subset \ell^\infty(G) $ for all $A\in \mathcal{S} $. Equivalently, a function  $ f\in \ell^\infty(G) $ belongs to $ I_{\mathcal{S}} $ if and only if for each $\varepsilon>0$, $\{g\in G: |f(g)|>\varepsilon\} $ is small relative to $ \mathcal{S} $.
\end{defn}

Note that every boundary piece $I$ can be written as $ I=I_{\mathcal S} $ for some $ \mathcal{S} $. Indeed, we can simply take $ \mathcal{S}:= \{ F\subset G: 1_F\in I \} $, where $1_F$ is the characteristic function supported on $F$.

When the family $\mathcal S$ is given by a family of subgroups $\mathcal S= \mathcal P=\{H_\lambda\}_{\lambda\in \Lambda}$, the boundary piece $c_0(G,\mathcal{P})\coloneqq  I_{\mathcal{S}} $ generated by $ \mathcal P $ is precisely the subgroup boundary piece defined in \cite[Definition 15.1.1]{BO08}.

\begin{defn}\label{relbiexactdef}
    Let $G$ be a discrete group and $ I\subset \ell^\infty(G) $ be a left and right $G$-invariant subalgebra. We say that $G$ is \emph{biexact relative to $I$} if the two-sided $G\times G$ action on $ \ell^\infty(G)/I $ is topologically amenable in the sense of Anantharaman-Delaroche.
    
    In particular, we say $G$ is \emph{biexact relative to a family of subgroups $ \mathcal{P}= \{H_\lambda\}_{\lambda\in \Lambda}$} if $G$ is biexact relative to $I=c_0(G,\mathcal P)$.
\end{defn}

We also recall the definition of an amenable action in the case where $\Gamma$ is a discrete group and $X$ is a locally compact Hausdorff space. The action $\Gamma\curvearrowright X$ is said to be amenable if there exists a net of continuous maps
\[
m_i\colon X\to \operatorname{Prob}(\Gamma)
\]
such that, for every $g\in\Gamma$,
\[
\|m_i(gx)-g\cdot m_i(x)\|_1\longrightarrow 0
\]
uniformly for $x$ in compact subsets of $X$. Here,
$\operatorname{Prob}(\Gamma)$ denotes the space of probability measures on $\Gamma$, and $\Gamma$ acts on $\operatorname{Prob}(\Gamma)$ by left translation. For an abelian $C^*$-algebra $C=C_0(X)$ with spectrum $X$, we also say a action of $ \Gamma $ on $ C=C_0(X) $ is amenable if the induced action on $X$ is amenable.

\subsection{Arrays on relative hyperbolic groups} One approach to studying relative bi-exactness on groups is through the more geometric notion of an array \cite{CS11}, which generalizes the classical concept of group length function.

\begin{defn}\label{defn: array}
Given a unitary representation $\pi: G \rightarrow \mathscr U(\mathcal H)$, we say that a map $\ur: G \rightarrow \mathcal H$ is an \emph{array} if it satisfies the following \emph{bounded equivariance} condition: for every $h,k\in G$,
\begin{equation}
\sup_{g\in G} \|\ur(hgk)-\pi_h(\ur(g))\|_{\mathcal H}<\infty.
\end{equation}

\end{defn}

If $I$ is a boundary piece of $G$, we say that $\ur$ is \emph{proper relative to} $I$ if for every $c\geq 0$ the $c$-radius ball with respect to $\ur$
\begin{equation}\label{relprop}
B^{\ur}_c\coloneqq  \{ g\in G \,:\, \|\ur(g)\|_{\mathcal H}\leq c\} \text{ is small relative to } I,
\end{equation}
i.e. small relative to the family of subsets $\mathcal{S}:=\{ F\subset G: 1_F\in I \} $.

When $I=c_0(G)$, the array $\ur$ is simply called \emph{proper}, since in this case \eqref{relprop} implies that all finite-radius balls $B^{\ur}_c$ are finite sets.

This notion is closely connected to relative bi-exactness; indeed, in \cite{CSU13,PV12} it was shown the following
\begin{thm}\label{thm: biexact=properarrays} $G$ is bi-exact relative to $I$ if and only if $G$ is exact and it admits an array with values in a weakly-$\ell^2$ representation that is proper relative to $I$.
\end{thm}

\medskip

Using the Mineyev-Yaman bicombing \cite{MY09}, Oyakawa was able to show in \cite{Oya23} that any relative hyperbolic group $G$ admits a class of arrays which witness a weak, yet natural, form of relative biexactness with respect to its canonical peripheral structure $\mathcal P$.

\begin{thm}[{\cite[Proposition~3.1, Remark ~3.2]{Oya23}}]\label{array-from-the-group}
    Let $G$ be a finitely generated group that is hyperbolic relative to a collection $\mathcal P=\{H_\lambda\}_{\lambda \in \Lambda}$ of subgroups of $G$. Then there exist a finite generating set $X$ of $G$, a 2-vertex-connected fine hyperbolic graph $\mathcal G$ on which $G$ acts without edge inversion and an array $\textbf{r}:G\rightarrow (\ell^2\mathcal G,\pi_{G})$ satisfying the following: 
    \begin{enumerate}
        \item The edge stabilizer is trivial for any edge $E(\mathcal G)$.
        \item For any $g\in G$, we get $d_{\tilde{\Gamma}}(1,g)\leq \frac{1}{2}\|\textbf{r}(g)\|^2_2$,
        where $\tilde{\Gamma}$ is the coned-off Cayley graph of $G$ with respect to $X$.
        \item The representation $(\ell^2\mathcal G,\pi_{G})$ is weakly-$\ell^2$. 
    \end{enumerate}

\end{thm}

This leads to a specific family of arrays on $G$ whose finite-radius balls have a structure that is very close to relative properness. 

\begin{cor}[{\cite[Proposition 3.2]{AMAKCK}}]\label{array-from-the-group} Let $G$ be a hyperbolic group relative to a finite collection of subgroups $\{H_1,\ldots, H_k \}$. Then there exist a weakly-$\ell^2$ representation $\sigma: G\rightarrow \mathscr U(\mathcal H)$ and an array $\textbf{r} :G\rightarrow \mathcal H$ such that for every $c\geq 0$ there exists $d_c>0$ for which we have $B^{\textbf{r}}_c\coloneqq \{ g\in G: \|\textbf{r}(g)\|\leq c\}\subseteq H_{i_1}^{a_1}H_{i_2}^{a_2}\cdots H_{i_k}^{a_k}F$ for some finite subset $F\subseteq G$ where $|F|,|a_j|<d_c$. Here $H_i^a:= aHa^{-1}  $.    
\end{cor}

This result was instrumental in Oyakawa's proof that relatively hyperbolic groups with bi-exact peripheral subgroups are themselves bi-exact. As we shall see shortly, it will also play a key role in the proof of our main theorem throughout the remainder of this section.

\vskip 0.1in 
 In preparation of the proof of our main result Theorem \ref{thm: main} we introduce more relevant notation. Fix a group $G$ together with peripheral subgroups $\mathcal{P}\coloneqq \{H_{\lambda}\}_{\lambda\in \Lambda}$. For every integer $k\geq 0$ we define the boundary piece $J_k=J_k(\mathcal{P})\subset \ell^\infty(G)$ to be the the boundary piece generated by $ \mathbb A_k\coloneqq \{ a_0H_{\lambda_1}a_1H_{\lambda_2}a_2\cdots a_{k-1}H_{\lambda_k}a_k\,|\, a_0,\ldots,a_k \in G,\lambda_i\in \Lambda\} $. By construction, we have  $\mathbb A_{k} \subset \mathbb A_{k+1}$ for all $k$. In particular, we have the following tower
\begin{equation}\label{filtration} c_0(G,\mathcal{P})\eqqcolon J_1\subset J_2\subset \cdots \subset J_k\subset \cdots \subset J_\infty \coloneqq \overline{ \bigcup_k J_k }\subset \ell^\infty(G).\end{equation}
The tower \eqref{filtration} should be viewed as a grading of the peripheral structure $\mathcal P$ of  $G$.

\begin{cor}\label{cor: relative biexact to Jinf}
     Let $G$ be a finitely generated group that is hyperbolic relative to a finite family of  subgroups $\mathcal{P}=\{H_\lambda\}_{\lambda\in \Lambda}$, then $G$ is biexact relative to $ J_\infty $.
\end{cor}
\begin{proof}
    This follows directly from Corollary \ref{array-from-the-group} and Theorem \ref{thm: biexact=properarrays}.
\end{proof}


To this end we notice that biexactness of $G$ relative to $ \mathcal{P}$ is equivalent to the $G\times G$-action on $ \ell^\infty(G)/J_1 $ being amenable. On the other hand, Corollary \ref{cor: relative biexact to Jinf} already shows that the $G\times G$-action on $ \ell^\infty(G)/J_\infty $ is amenable. Therefore, since $ J_\infty$ is the inductive limit of $J_k$'s, to show that $G$ is biexact relative to $\mathcal{P}$, it suffices to show that the $ G\times G $ action on $ J_k/J_{k-1} $ is amenable for each $k\geq 2$. 

\subsection{Decomposition of $ \text{Sp}(J_k/J_{k-1}) $ for hyperbolically embedded subgroups}
Similar to the upgrading argument for an almost malnormal biexact subgroup in \cite{KEY26}, to show amenability of the $G\times G$ action on $ J_k/J_{k-1}$, we first decompose its spectrum as a disjoint union of subsets of the form $ \text{Sp}(\ell^\infty( T )/\ell^\infty(T)\cap J_{k-1}) $, where $ T= a_0H_{\lambda_1}a_1\cdots a_{k-1}H_{\lambda_k}a_k \in \mathbb A_k$.

First notice that the separation property of the elements of $\mathbb A_k$, established in Theorem \ref{thm: decomposition of Jk}, yields the following key property of $J_k$  for our later arguments.  

\begin{thm}\label{orthproj}
    Let $k\geq 1$, and $T=a_0H_{\lambda_1}a_1H_{\lambda_2}a_2\cdots a_{k-1}H_{\lambda_k}a_k, \quad S=b_0H_{\mu_1}b_1\cdots b_{k-1}H_{\mu_k}b_k\in \mathbb A_{k}$ with $a_i,b_i\in G$. If $ p_T,p_S\in J_k/J_{k-1} $ denote the characteristic function of $ T $ and $S$, then either $ p_T=p_S $ or $ p_Tp_S=0 $.
\end{thm}

\begin{proof}
    This follows directly from Theorem \ref{thm: decomposition of Jk}.
\end{proof}

From the previous theorem, we immediately obtain two corollaries.

\begin{cor}
    \[\text{Sp}(J_k/J_{k-1}) = \bigsqcup_{ T\in \mathbb A_{k}\backslash \mathbb A_{k-1} } \text{Sp}(\ell^\infty(T)/\ell^\infty(T)\cap J_{k-1}). \] 
\end{cor}
\begin{proof}
 Using Theorem \ref{orthproj} one can see that \[J_k/J_{k+1}=\bigoplus_{T\in \mathbb A_k\setminus \mathbb A_{k-1}} p_T (J_k/J_{k-1})= \bigoplus_{T \in \mathbb A_k\setminus \mathbb A_{k-1} } \left (\ell^\infty (T)/\ell^\infty(T)\cap J_{k-1}\right).\]
 Then the conclusion follows from Gelfand duality.
\end{proof}

\begin{cor}\label{cor: stabilizer} The natural action of $G\times G$ on $Sp(J_k/J_{k-1})$ permutes the components $\{\text{Sp}(\ell^\infty(T)/\ell^\infty(T)\cap J_{k-1})\}_{T \in \mathbb A_k \setminus \mathbb A_{k-1}}$.
Moreover, if  $ T= a_0H_{\lambda_1}a_1\cdots a_{k-1}H_{\lambda_k}a_k\in \mathbb A_k\backslash \mathbb A_{k-1} $, the stabilizer of $ \text{Sp}(\ell^\infty(T)/\ell^\infty(T)\cap J_{k-1}) \subset \text{Sp}(J_k/J_{k-1}) $ under this action is
    \[ (a_0H_{\lambda_1}a_0^{-1})\times (a_k^{-1}H_{\lambda_k}a_k). \]
    
\end{cor}
\begin{proof} Since $G\times G$ permutes the sets $\mathbb A_k$, the natural   $G\times G$-action also permutes the components $ \{\text{Sp}(\ell^\infty(T)/\ell^\infty(T)\cap J_{k-1})\}_{T\in \mathbb A_k\setminus \mathbb A_{k-1}}   $. 

    Suppose $ (g_1,g_2) $ is inside the stabilizer, then it should preserve the identity of $ \ell^\infty(T)/\ell^\infty(T)\cap J_{k-1} $. In particular, $g_1Tg_2^{-1}= g_1a_0H_{\lambda_1}\cdots H_{\lambda_k}a_kg_2^{-1}$ should have non $ \mathbb A_{k-1}$-small intersection with $T$. Using the proof of Theorem \ref{thm: decomposition of Jk}, we must have $ g_1a_0H_{\lambda_1}=a_0H_{\lambda_1} $ and $ H_{\lambda_k}a_kg_2^{-1}=Ha_k $, implying that $ g_1\in a_0H_{\lambda_1}a_0^{-1} $ and $ g_2\in a_k^{-1}H_{\lambda_k}a_k $. 
\end{proof}

\subsection{Gluing and exactness of $H_\lambda$'s.}

We can use the following standard lemma to show the amenability of the $G\times G$ action on $\text{Sp}(J_k/J_{k-1})$.
\begin{lemma}\label{lem: gluing}
    Let $\Lambda$ be a discrete group acting on a locally compact Hausdorff space $\Pi$ such that $ \Pi=\bigsqcup_{i\in I} \Pi_i $ with each $\Pi_i$ an open compact subspace. Suppose $\Lambda$ permutes the $\Pi_i$'s, and for each $i$, the stabilizer $\Lambda_i \coloneqq \{\lambda\in \Lambda: \lambda(\Pi_i) =\Pi_i\}$ acts amenably on $\Pi_i$, then the action of $\Lambda$ on $ \Pi $ is also amenable.
\end{lemma}
\begin{proof}
    Let $\Lambda\curvearrowright I$ denote the induced action on the index set corresponding to the permutation action of $\Lambda$ on $\{\Pi_i\}_{i\in I}$, and let $I/\Lambda$ denote the set of $\Lambda$-orbits in the index set $I$. For each orbit $O\in I/\Lambda$, choose a representative $i_O\in O$. For every $j\in O$, choose an element $s(j) \in \Lambda$ such that $s(j) \Pi_{i_O} = \Pi_j$, and we pick $s(i_O) = e$.

    Since $\Lambda_{i_O}\curvearrowright \Pi_{i_O}$ is amenable, there exist a directed set $D_O$ and a net of continuous maps $m_{\alpha}^O \colon \Pi_{i_O}\to \text{Prob}(\Lambda_{i_O})$, $\alpha\in D_O$ such that for every $\lambda \in \Lambda_{i_O}$,
    \begin{equation} \label{eqn: amen action}
        \lim_{\alpha\to\infty} \sup_{x\in \Pi_{i_O}} \lVert \lambda\cdot m_{\alpha}^O(x) - m_{\alpha}^O(\lambda\cdot x)\rVert_1 =0.
    \end{equation}

    Let $A\coloneqq \prod_{O\in I/\Lambda}\alpha_O$ be the product directed set ordered coordinate-wise. For each $\alpha = (\alpha_O)_O \in A$, we define a continuous map $M_\alpha\colon \Pi\to \text{Prob}(\Lambda)$ piecewise on each $\Pi_j$ as follows: Let $O$ denote the orbit of $j\in I$. For each $x\in \Pi_j$, since $s(j)^{-1}x\in \Pi_{i_O}$, we define
    \[M_\alpha(x)\coloneqq s(j)\cdot m_{\alpha_O}^O (s(j)^{-1}\cdot x),\]
    where regard $\text{Prob}(\Lambda_{i_O})$ as a subset of $\text{Prob}(\Lambda)$ by extension by zero. Since each $\Pi_j$ is open in $\Pi$ and the definition on each $\Pi_j$ is continuous, $M_\alpha$ is continuous on $\Pi$.

    We now verify asymptotic equivariance. Fix $\lambda\in \Lambda$ and $x\in \Pi_j$. Let $O\subset I$ be the orbit of $j$, so then $\lambda\cdot j\in O$ as well. Note that $s(\lambda\cdot j)^{-1}\lambda s(j)\in \Lambda_{i_O}$ since both $s(\lambda\cdot j)$ and $\lambda s(j)$ send $\Pi_{i_O}$ onto $\Pi_{\lambda\cdot j}$. We compute
    \[\lambda\cdot M_\alpha(x) 
    = \lambda\cdot s(j)\cdot m_{\alpha_O}^O(s(j)^{-1}\cdot x)
    =  s(\lambda\cdot j)\cdot \big((s(\lambda\cdot j)^{-1}\lambda s(j)\cdot m_{\alpha_O}^O(s(j)^{-1}\cdot x)\big),\]
    while
    \[M_\alpha(\lambda\cdot x) 
    = s(\lambda\cdot j)\cdot m_{\alpha_O}^O(s(\lambda\cdot j)^{-1}\lambda \cdot x) 
    = s(\lambda\cdot j)\cdot m_{\alpha_O}^O\big(s(\lambda\cdot j)^{-1}\lambda s(j)s(j)^{-1}\cdot x\big),\]
    so for $x\in \Pi_j$,
    \[\lVert \lambda\cdot M_\alpha(x) - M_\alpha(\lambda\cdot x)\rVert_1 = \left\Vert \big(s(\lambda\cdot j)^{-1}\lambda s(j)\big)\cdot m_{\alpha_O}^O(s(j)^{-1}\cdot x) - m_{\alpha_O}^O\big(\big(s(\lambda\cdot j)^{-1}\lambda s(j)\big) s(j)^{-1}\cdot x\big)\right\Vert_1.\]
    
    Let $K\subset \Pi$ be compact. Since $\Pi = \sqcup_i \Pi_i$ is a disjoint union with each $\Pi_i$ open, $K$ only intersect with finitely many $\Pi_j$'s, so the index set $F_K\coloneqq \{j\in I: K\cap \Pi_j\neq\varnothing\}$ is finite. For each $j \in F_K$, the set $K_j\coloneqq s(j)^{-1}(K\cap \Pi_j) \subset \Pi_{i_O}$ is compact. Hence by (\ref{eqn: amen action}),
    \begin{align*}
        &\lim_\alpha \sup_{x\in K}\lVert \lambda\cdot M_\alpha(x) - M_\alpha(\lambda\cdot x)\rVert_1 \\=& \lim_\alpha\sup_{j\in F_K} \sup_{x\in K\cap \Pi_j} \left\Vert \big(s(\lambda\cdot j)^{-1}\lambda s(j)\big)\cdot m_{\alpha_O}^O(s(j)^{-1}\cdot x) - m_{\alpha_O}^O\big(\big(s(\lambda\cdot j)^{-1}\lambda s(j)\big) s(j)^{-1}\cdot x\big)\right\Vert_1 = 0,
    \end{align*}
    since $F_k$ is a finite set and there could be only finitely many orbits $O$ and group elements of the form $s(\lambda\cdot j)^{-1}\lambda s(j)\big) s(j)^{-1}$ appearing for $j\in F_k$. This proves $\Lambda\curvearrowright \Pi$ is amenable as desired.
\end{proof}

\begin{lemma}\label{subgroup amen}
    Suppose that each $H_{\lambda}$ is exact. For any $T= H_{\lambda_1}a_1H_{\lambda_2}\cdots H_{\lambda_{k-1}}a_{k-1}H_{\lambda_k}\in \mathbb A_k/\mathbb A_{k-1}$, the $H_{\lambda_1}\times H_{\lambda_{k}}$ action on $ \text{Sp}(\ell^\infty(T)/\ell^\infty(T)\cap J_{k-1}) $ is amenable.
\end{lemma}
\begin{proof}
    As $\ell^\infty(T)/\ell^\infty(T)\cap J_{k-1}$ is a quotient of $\ell^\infty (T)$, then $\text{Sp}(\ell^\infty(T)/\ell^\infty(T)\cap J_{k-1})$ is a closed invariant subset of $ \text{Sp}(\ell^\infty(T)) $. Since amenability passes to restrictions to closed invariant subsets it suffices to show that the $H_{\lambda_1}\times H_{\lambda_{k}}$ action on $\beta T\coloneqq \text{Sp}(\ell^\infty(T)) $ is amenable. Now, by Theorem \ref{Prop:maln}, for every $g\in T\subset G\backslash \cup_\lambda H_\lambda$ its stabilizer  ${\rm Stab} (g)= H_{\lambda_1}\cap gH_{\lambda_k}g^{-1}$ is finite. Since $H_{\lambda_1}\times H_{\lambda_{k}}$ is exact and acting on the discrete set $T$ with finite stabilizer, this implies that the $H_{\lambda_1}\times H_{\lambda_{k}}$ action on $\beta T$ is amenable. (Indeed, the finite stabilizer condition gives us a $G\times G$-equivariant u.c.p. map $ \phi:\ell^\infty(G\times G) \to \ell^\infty(T)$, $ \phi(f)((g_1,g_2)x)\coloneqq \tfrac{1}{|\text{Stab}(x)|}\sum_{(h_1,h_2)\in \text{Stab}(x)}f( g_1h_1,g_2h_2 ) $ where $x\in T$ is a chosen orbit representative.)
\end{proof}

\begin{thm}\label{thm: hyperbolic embedding amenable action}
    If $\{H_\lambda\}_{\lambda \in \Lambda} \hookrightarrow_{h} (G,X) $, and $ H_\lambda $ is exact for all $\lambda \in \Lambda$, then the $G\times G$ action on $ J_\infty/J_1 $ is amenable.
\end{thm}
\begin{proof}
    Since $J_\infty$ is the inductive limit of $J_k$'s, it suffices to show that $J_k/J_{k-1}$ is $G\times G$ amenable for each $k\geq 2$, but this follows from Lemma \ref{subgroup amen}, Corollary \ref{cor: stabilizer}, and the gluing Lemma \ref{lem: gluing}.
\end{proof}

\begin{thm}\label{thm: main}
    If $G$ is finitely generated and hyperbolic relative to a finite family of exact subgroups $\mathcal{P}=\{H_{\lambda}\}_{\lambda \in \Lambda}$, then $G$ is biexact relative to $\mathcal{P}$.
\end{thm}
\begin{proof}
    This follows directly from Theorem \ref{thm: hyperbolic embedding amenable action} and Corollary \ref{cor: relative biexact to Jinf}. 
\end{proof}

\noindent{\bf Remarks}. Taken together, Theorems~\ref{thm: biexact=properarrays} and~\ref{thm: main} imply that every exact group that is relatively hyperbolic admits weakly-$\ell^2$ arrays that are proper relative to its natural peripheral structure. At present, however, we do not know of a direct geometric construction of such arrays, and we leave this as an open problem.

\section{Relative Solididity and Primeness Results}

\subsection{Popa's intertwining techniques}\label{popaint}  In \cite [Theorem 2.1 and Corollary 2.3]{Po03} Popa introduced a powerful analytic criterion for identifying intertwiners between arbitrary subalgebras of tracial von Neumann algebras, see Theorem \ref{corner} below. This technique,  known as \emph{Popa's intertwining technique},  has played an essential role in the classification of von Neumann algebras program via Popa's deformation/rigidity theory.  

\begin{thm}\emph{\cite{Po03}} \label{corner} Let $( \mathcal{M},\tau)$ be a  tracial von Neumann algebra and let $ \mathcal{P},  \mathcal{Q}\subseteq  \mathcal{M}$ be (not necessarily unital) von Neumann subalgebras. 
	Then the following are equivalent:
	\begin{enumerate}
		\item There exist projections $ p\in    \mathcal{P}, q\in    \mathcal{Q}$, a $\ast$-homomorphism $\theta:p  \mathcal{P} p\rightarrow q \mathcal{Q} q$  and a partial isometry $0\neq v\in  \mathcal{M} $ such that $v^*v\leqslant p$, $vv^*\leqslant q$ and $\theta(x)v=vx$, for all $x\in p  \mathcal{P} p$.
		\item For any group $\mathcal G\subset \mathscr U( \mathcal{P})$ such that $\mathcal G''=  \mathcal{P}$ there is no net $(u_n)_n\subset \mathcal G$ satisfying $\|E_{  \mathcal{Q}}(xu_ny)\|_2\rightarrow 0$, for all $x,y\in   \mathcal{M}$.
		\item There exist finitely many $x_i, y_i \in  \mathcal{M}$ and $C>0$ such that  $\sum_i\|E_{  \mathcal{Q}}(x_i u y_i)\|^2_2\geqslant C$, for all $u\in \mathscr U( \mathcal{P})$.

	\end{enumerate}
\end{thm} 
\vskip 0.02in
\noindent If one of the three equivalent conditions from Theorem \ref{corner} holds, then we say that \emph{ a corner of $ \mathcal{P}$ intertwines into $ \mathcal{Q}$ inside $ \mathcal{M}$}, and write $ \mathcal{P}\prec_{ \mathcal{M}} \mathcal{Q}$.

For further use we recall the following two basic lemmas on intertwining.

\begin{lemma}[Lemma 2.2 in \cite{CI18}]\label{intsubgroups} Let $H,K<G$ be groups, let $G \curvearrowright \mathcal N$ be a trace preserving action on a tracial von Neumann algebras $\mathcal N$. Denote by $\mathcal M = \mathcal N\rtimes G$ the corresponding crossed product von Neumann algebra. If $\mathcal L(H)\prec_{\mathcal M}\mathcal N \rtimes K$ then one can find $g\in G$ such that $[H: H\cap gKg^{-1}]<\infty$.    
\end{lemma}

\begin{lemma}\label{passtofiniteindexext}
 Let $\sP,\sQ, \sR\subseteq \sM$ be tracial von Neumann algebras such that $\sQ\subseteq \sR$ has finite index. If $\sQ\prec_\sM \sP$ then also $\sR\prec_\sM \sP$.     
\end{lemma}

\subsection{Relative solidity results}

In this subsection, we briefly record several relative solidity results for crossed products arising from actions of relatively hyperbolic groups. These significantly strengthen \cite[Theorems 5.2--5.3]{AMAKCK} by removing the residual finiteness assumption on the peripheral subgroups of the relatively hyperbolic group. Since the proofs are identical to those in \cite{AMAKCK}, we omit them. Indeed, they carry over verbatim, with the only modifications being the replacement of \cite[Theorem 5.1]{AMAKCK} by Theorem~\ref{thm: main}. 
\vskip 0.05in
To state these results, we first recall the notion of the one-sided quasi-normalizer. For an inclusion of von Neumann algebras $\sA\subseteq \sM$, let $\mathcal{QN}^1_\sM(\sA)=\{x\in \mathcal M \,:\, \text{ exist } x_1, \ldots, x_n \in\mathcal M \text{ such that } \mathcal A x\subseteq \sum_i x_i\mathcal A  \}$ denote the set of one-sided quasi-normalizers of $\sA$ in $\sM$. Also its quasi-normalizer is defined as $\mathcal {QN}_{\mathcal M}(\mathcal A)= \mathcal{QN}^1_{\sM}(\sA)\cap (\mathcal{QN}^1_{\sM}(\sA))^*$. 

Now let  $\mathcal{A}^0=\mathcal{A}$. We then define, by transfinite induction, a tower of von Neumann algebras associated with the one-sided quasi-normalizer. If $\beta$ is a successor ordinal, define $\mathcal{A}^{\beta}= \mathcal {QN}^1_\sM(\sA^{\beta-1})''$. Observe that $\mathcal{A}^{\beta-1}\subseteq \mathcal{A}^{\beta}$. If $\beta$ is a limit ordinal, let $\mathcal{A}^{\beta}=\overline{\bigcup_{\gamma<\beta}\sA^\gamma}^{WOT}$. Finally, let $\alpha$ be the least ordinal for which the chain $(\mathcal{A}^{\beta})_{\beta}$ stabilizes, that is, $\mathcal{A}^{\alpha}=\mathcal{A}^{\alpha+1}$. We denote this terminal von Neumann algebra by $\overline{\mathcal{QN}}_{\sM}^{1}(\mathcal{A})$.

\begin{thm}\label{commut1} Let $G$ be a finitely generated, non-elementary group that is hyperbolic relative to exact proper subgroups $\{ H_1, \ldots, H_n \}.$  Let $G \curvearrowright \mathcal N$ be a trace preserving action on an amenable von Neumann algebra and denote by $\mathcal{M}=\mathcal N \rtimes G$ be the corresponding cross-product von Neumann algebra.  Let $p \in \sM$ be a nonzero projection and let $\mathcal{A} \subseteq p\mathcal{M}p$ be a diffuse von Neumann subalgebra such that $\mathcal{A}' \cap p\mathcal{M}p$ has no amenable direct summand. Let $\mathcal{Q} = \mathcal{QN}_{p\mathcal{M}p}(\mathcal{A})^{''}.$ Then one can find $i\in \overline{1,n}$, projections $r \in  \mathcal{A}, q \in \mathcal{A}' \cap p\mathcal{M}p$ with $ rq \neq 0$ and $u \in \mathscr{U}(\mathcal{M})$ such that \begin{equation}\label{cont3}u(rq \overline{\mathcal{QN}}_{p\mathcal{M}p}^{1}(\mathcal{Q})r q)u^* \subseteq \mathcal N \rtimes H_i.\end{equation}
If in addition $\sN\rtimes H_i$ are factors then in \eqref{cont3} we can replace $r q$ by its central support in $\overline{\mathcal{QN}}_{p\mathcal{M}p}^{1}(\mathcal{Q}).$
    
\end{thm}

\begin{thm}\label{commut2}
 Let $G$, $G \curvearrowright \mathcal N$ and $\mathcal{M}=\mathcal N \rtimes G$ be as in Theorem \ref{commut1}.    Let $p \in \sM$ be a nonzero projection and let $\mathcal{A} \subseteq p\mathcal{M}p$ be a von Neumann subalgebra such that $\mathcal A \nprec_\sM \mathcal N$.  Denote by $\mathcal{Q}=\mathcal{QN}_{p\mathcal{M}p}(\mathcal{A})^{''}.$ Then one can find orthogonal projections $p_0,  \dots, p_n \in \mathscr{Z}(\overline {\mathcal{QN}}_{p \sM p}^{1}(\mathcal{Q}))$ with $p_0 + \dots + p_n = p$ that are maximal with the following properties: 
    \begin{enumerate}
        \item $(\mathcal{A}^{'} \cap p\mathcal{M}p)p_0$ is amenable.
        \item Any corner of $\overline {\mathcal{QN}}_{p\sM p}^{1}(\mathcal{Q})p_i$ has a nontrivial subcorner which can be unitarily conjugated into $\mathcal N \rtimes H_i$, for every $i=\overline{1,n}$.
         \item If in addition $\mathcal N \rtimes  H_i$ are factors then one can find $u_1, \ldots, u_n \in \mathscr U(\mathcal{M})$ such that $u_i \overline{\mathcal{QN}}_{\mathcal{M}}^{1}(\mathcal{Q})p_i u_i^{*} \subseteq \mathcal N \rtimes  H_i$, for every $i=\overline{1,n}.$ In particular, when $n=1$ and $\sA'\cap p\sM p$ has no amenable direct summand there is $u\in \mathscr U(\mathcal M)$  such that \begin{equation*}u \overline{\mathcal{QN}}_{\mathcal{M}}^{1}(\mathcal{Q}) u^{*} \subseteq \mathcal N \rtimes  H_1.\end{equation*} 
\end{enumerate}
\end{thm}

\begin{cor}
    If $G$ is finitely generated, non-elementary and hyperbolic relative to a finite family of proper exact subgroups $\mathcal{H}=\{H_{\lambda}\}_{\lambda \in \Lambda}$, then $L(G)$ is s-prime. That is, for any nonzero projection $p\in L(G)$ there are  no commuting diffuse von Neumann subalgebras $\sP, \sQ\subset pL(G)p$ which generate together a von Neumann algebra $\sP\vee \sQ\subseteq pL(G)p$ of finite index, $[pL(G)p: \sP\vee\sR]<\infty$. In particular, $L(G)$ is prime.
\end{cor}
\begin{proof} Assume by contradiction that such $\sP,\sQ\subseteq pL(G)p$ exist. Since $G$ is nonamenable the finite index assumption implies that so is $\sP\vee\sR$. Further cutting by a nonzero subprojection of $p$ we can further assume that either $\sP$ or $\sR$  has no amenable direct summand. Thus Theorem \ref{commut2} implies there exists a peripheral subgroup $H_i\in \mathcal H$ such that nontrivial corners of $\sP\vee \sR$ can be unitarily conjugated in $L(H_i)$. By Lemma \ref{passtofiniteindexext}, the finite index assumption further yields $L(G )\prec_{L(G)} L(H_i)$. Using Lemma \ref{intsubgroups} we get $[G:H_i]<\infty$, which is a contradiction.
\end{proof}

\noindent{\bf Remarks.} We note in passing that primeness of $L(G)$ can also be deduced directly from \cite[Proposition 4.11]{Oya24} and Theorem \ref{thm: main}. 

\section{W$^*$-rigidity Results for Actions of Property (T) Groups}\label{sec: W*-rigidity}
This section constructs a continuum of property (T) groups satisfying the main W$^*$-rigidity result stated in Theorem~\ref{mainresult:nonW^*e}. Our approach combines the structural results established in the previous sections with techniques from the theories of Cartan subalgebras and measure equivalence. We begin by outlining these ingredients, which are developed in the following subsections.

\subsection{Uniqueness of Cartan subalgebras}

The first is a unique Cartan subalgebra theorem for actions of certain relatively hyperbolic groups. Its first technical ingredient is the following control of quasinormalizers theorem for crossed-product von Neumann algebras.

\begin{thm}\label{intmalnormal}
Let $H<G$ be groups such that $H$ is malnormal in $G$. Let $G\curvearrowright\mathcal B$ be a trace-preserving action on a tracial von Neumann algebra and let $\mathcal M=\mathcal B\rtimes G$ be the corresponding crossed product von Neumann algebra. Let $p\in\mathcal M$ be a projection, and let $\mathcal Q\subseteq p\mathcal Mp$ be a von Neumann subalgebra such that $\mathcal Q\prec_{\mathcal M}\mathcal B\rtimes H$ and $\mathcal Q\nprec_{\mathcal M}\mathcal B$. Then there exist a unitary $u\in\mathcal M$ and projections $e\in\mathcal Q$ and $r\in\mathcal Q'\cap p\mathcal Mp$ such that
\begin{equation}\label{cont2}
ue\,\mathcal{QN}''_{p\mathcal Mp}(\mathcal Q)\,eru^*\subseteq\mathcal B\rtimes H.
\end{equation}
In particular, every unital von Neumann subalgebra $\mathcal P\subseteq\mathcal{QN}''_{p\mathcal Mp}(\mathcal Q)$ satisfies $\mathcal P\prec_{\mathcal M}\mathcal B\rtimes H$.
\end{thm}


    \begin{proof}
        Since $\mathcal Q\prec_{\mathcal M} \mathcal B \rtimes H$, there exist nonzero projections $e\in \mathcal Q, f\in \sB\rtimes H$, a partial isometry $v\in \mathcal M$ and a unital $\ast$-homomorphisms $\theta: e \mathcal Q e \rightarrow \mathcal \theta(e\sQ e)=:\sR\subseteq f(\mathcal B\rtimes H )f$ such that 
\begin{equation}\label{intrel1}
  \theta(x)v=vx, \text{ for all } x\in e\sQ e.
\end{equation}

Set $s\coloneqq vv^*\in \sR'\cap f \sM f$ and $  r\coloneqq v^*v\in e\sQ e'\cap e\mathcal M e$. Since $\sQ \nprec_\sM \sB$, transitivity of intertwining \cite{IPP05} yields $\sR \nprec_{\sB \rtimes H}\sB$. 
 Hence by \cite{Po03} (see also \cite[Appendix A]{Bo14}) we have \begin{equation}\label{cont1}\mathcal {QN}''_{f \mathcal M f} (\sQ) \subseteq \mathcal B\rtimes H.\end{equation} Let $u\in $ be unitary such that $v= u r$. Relation \eqref{intrel1} implies that $u e\sQ e u^*= \sR s$. Hence using \eqref{cont1} and the compression formulae for quasinormalizing algebras   from \cite{Po03,FGS10}, we obtain \begin{equation*} \begin{split} &  \mathcal B \rtimes H \supseteq s \mathcal {QN}''_{f\mathcal M f}(\sR)s = \mathcal{ QN}''_{s \mathcal M s}(\sR s)= \mathcal {QN}''_{u r \mathcal M r u^* }(ue\sQ e r u^*)= u er\mathcal {QN}''_{  p \mathcal M p}(\sQ) er u^*,
 \end{split}\end{equation*}
which proves \eqref{cont2}.

Finally, \eqref{cont2} immediately implies that $\mathcal {QN}''_{p\sM p}(\sQ)\prec_\sM \sB\rtimes H$.  The conclusion now follows from \cite[Lemma 3.4]{Va08}, which shows that every unital von Neumann subalgebra $\mathcal P\subseteq\mathcal{QN}''_{p\mathcal Mp}(\mathcal Q)$ satisfies $\mathcal P\prec_{\mathcal M}\mathcal B\rtimes H$.
\end{proof}

\begin{thm}\label{uniquecartan}
   For every $i=1,2$, let $G_i$ be a finitely generated group that is hyperbolic relative to an icc non-amenable subgroup $H_i<G_i$. Assume that $H_i =H_i^1\times H^i_2$ where $H_j^i$ are infinite. Assume that $G_i \curvearrowright Z_i$ is a free ergodic pmp action on a standard probability space and denote by $\mathcal M_i = L^\infty(Z_i)\rtimes G_i$. 
   
   If $p\in L^{\infty}(Z_1)$ is any projection so that $p \sM_1 p= \sM_2$ then one can find a unitary $u\in \sM_2$ such that $u L^\infty(Z_1)p u  = L^\infty (Z_2)$. In particular, we have  $G_1 \cong_{\rm ME} G_2$. \end{thm}

\begin{proof}  Set \(\mathcal A=L^\infty(Z_1)\) and \(\mathcal B=L^\infty(Z_2)\). Then \(\mathcal M_1=\mathcal A\rtimes G_1\) and \(\mathcal M_2=\mathcal B\rtimes G_2\). From the hypothesis, we may assume that there is a nonzero projection \(p\in \mathcal A\) such that \(p\mathcal M_1p=\mathcal M_2\). We first prove that 
\begin{equation}
    \mathcal B\prec_{\mathcal M_2} \mathcal A. 
\end{equation}
Since $L(H_1^1)$ and $L(H_1^2)$ are commuting von Neumann subalgebras of $\mathcal M_1$, and at least one of them (say $L(H_1^2)$) has no amenable direct summand, we obtain by Theorem \ref{commut1} that $L(H_1^1)\prec_{\mathcal M_1} \mathcal A\rtimes H_1$. Moreover, since $H_1^2$ is icc, the algebra $L(H_1^2)$ is a II$_1$ factor, and since $\mathcal A$ is abelian we have $L(H_1^2)\nprec_{\mathcal M_1} \mathcal A$. Hence, since $H_2<G_2$ is almost malnormal by Proposition \ref{Prop:maln}, using Theorem \ref{intmalnormal}, we get $\mathcal {QN}''_{\mathcal M_2}(L(H_1^2))\prec \mathcal A\rtimes H_1$, and since $L(H_2)\subseteq \mathcal {QN}''_{\mathcal M_2}(L(H_1^2))$ we obtain $L(H_2)\prec \mathcal A\rtimes H_1$. Thus, applying Theorem \ref{intmalnormal} again and repeating the same argument, we can find a projection $e\in L(H_2)$, $z\in L(H_2)' \cap \mathcal M_2 \subset \mathcal B^{H_2}$, and $u\in \mathscr U(\mathcal M_1)$ so that  

\begin{equation}
   u e L(H_2) e z u^*\subseteq \mathcal A \rtimes H_1.
\end{equation}
   Since $L(H_2)$ is a II$_1$ factor, after shrinking $e$ if necessary, we can assume that $\tau(e)=1/n$. Thus one can find $u_1,...,u_n\in L(H_2)$ with $\sum_i u_ieu_i^*=1$. Thus, for every $x\in L(H_2)z$, we have that $x= (\sum_i u_i e u_i^*)x(\sum_j u_i e u_j^*)= \sum_{i,j} u_i e( u_i^* x u_j) e u_j^*\in \sum_{i,j} u_i u (\mathcal A \rtimes G_1) u^*u_i^*$. Hence, letting $y_i = u_i u$ we get \begin{equation}
    L(H_2)z\subseteq \sum_{i,j} y_i \big(\mathcal A\rtimes H_1\big)y_i^*.
\end{equation} 
Using $\|\cdot\|_2$-approximations of $y_i$ we obtain the following: for every $\varepsilon>0$ there exists a finite set $F\subset (G_1\setminus H_1) \cup\{1\}$ such that \begin{equation}\label{smallsetsH_2}
    \sup_{x\in (L(H_2)z)_1}\|x- \mathscr P_{FH_1F}(x)\|_2\leq \varepsilon.
\end{equation}   
Here, and throughout this proof, for every subset $S\subseteq G_1$ we denote by $\mathscr P_{S}$ the orthogonal projection from $L^2(\mathcal M_1)$ onto $\overline{\text{span}\{au_g \,:\, a\in \mathcal A , g \in S\}}^{\|\cdot\|_2}$.

We next prove the following 

\begin{claim}\label{intcartan}
 We have $\mathcal B \prec_{\mathcal M_1} \mathcal A \rtimes H_1$.   
\end{claim}
\noindent \emph{Proof of Claim \ref{intcartan}}.  Fix $\varepsilon>0$. Using iteratively relation \eqref{smallsetsH_2} for every $i=\overline{1,4}$
pick finite subset $E_i=E_i^{-1}\subset (G_1 \setminus H_1)\cup \{1\}$ such that for all $g\in H_2$ we have 

\begin{equation}\label{smallsetsH_12}\begin{split}
   & \|u_g z -\mathscr P_{E_1 H_1 E_1}(u_g z)\|_2\leq \varepsilon, \text{ and}
   \\ 
   &\|u_g z -\mathscr P_{E_{i+1} H_1 E_{i+1}}(u_g z)\|_2\leq \frac{\varepsilon}{|E_1|^2\cdots |E_i|^2} \text{ for all } i=\overline{1,3}. 
   \end{split}
\end{equation}
Also note that for every $x\in \mathcal M_1$ we have \begin{equation}\label{strongnorm proj}\|\mathscr P_{E_i H_1 E_i}(x)\|_\infty \leq |E_i|^2 \|x\|_\infty.\end{equation}  

Next, since $\mathcal B$ is abelian and $\{u_gz\}_{g\in H_2}$  normalizes $\mathcal B z$, for all $a\in \mathscr U(\mathcal B)$ and all $g\in H_2$ we have $z=(az) (u_g z) (b z) (u_{g^{-1}} z) (a^* z) (u_gz) (b^*z)( u_{g^{-1}}z)$. Using this relation together with \eqref{smallsetsH_12}-\eqref{strongnorm proj} we see that 
\begin{equation}\label{traceineq1}
\begin{split}
    \tau(z)&=\tau((az) (u_g z) (a z) (u_{g^{-1}} z) (a^* z) (u_gz) (a^*z)( u_{g^{-1}}z))\\ 
    &\leq \tau((az) \mathscr P_{E_1 H_1E_1}(u_g z) (a z) (u_{g^{-1}} z) (a^* z) (u_gz) (a^*z)( u_{g^{-1}}z))+\|u_g z -\mathscr P_{E_1 H_1E_1}(u_gz)\|_2\\
    & \overset{ \eqref{smallsetsH_12} }{\leq} \tau((az) \mathscr P_{E_1 H_1E_1}(u_g z) (a z) (u_{g^{-1}} z) (a^* z) (u_gz) (a^*z)( u_{g^{-1}}z))+\varepsilon\\
    &\leq \tau((az) \mathscr P_{E_1 H_1E_1}(u_g z) (a z) \mathscr P_{E_2 H_1E_2}(u_{g^{-1}} z) (a^* z) (u_gz) (a^*z)( u_{g^{-1}}z))+\\& \quad +\|\mathscr P_{E_1H_1E_1}(u_g z)\|_\infty \|u_{g^{-1}z}-\mathscr P_{E_2 H_1E_2}(u_{g^{-1}}z)\|_2  + \varepsilon \\
    &\overset{\eqref{strongnorm proj}}{\leq} \tau((az) \mathscr P_{E_1 H_1E_1}(u_g z) (a z) \mathscr P_{E_2 H_1E_2}(u_{g^{-1}} z) (a^* z) (u_gz) (a^*z)( u_{g^{-1}}z))+\\& \quad +|E_1|^2 \|u_{g^{-1}z}-\mathscr P_{E_2 H_1E_2}(u_{g^{-1}}z)\|_2  + \varepsilon\\
    & \overset{\eqref{smallsetsH_12}}{\leq} \tau((az) \mathscr P_{E_1 H_1E_1}(u_g z) (a z) \mathscr P_{E_2 H_1E_2}(u_{g^{-1}} z) (a^* z) (u_gz) (a^*z)( u_{g^{-1}}z))+2 \varepsilon\\
    & \ldots\\
    &\leq\tau((az) \mathscr P_{E_1 H_1E_1}(u_g z) (a z) \mathscr P_{E_2 H_1E_2}(u_{g^{-1}} z)(a^* z) \mathscr P_{E_3H_1E_3}(u_gz) (a^*z)\mathscr P_{E_4H_1E_4}( u_{g^{-1}}z))+4 \varepsilon.
    \end{split}
\end{equation}

Assume for contradiction that $\mathcal B z \nprec \mathcal A \rtimes H_1$. By Popa's intertwining technique and basic $\|\cdot \|_2$-approximations this implies that for every $\varepsilon>0$ there is $a\in \mathscr U(\mathcal B)$ such that \begin{equation}\label{nonint4}
    \|\mathscr P_{E^{-1}_i H_1 E^{-1}_{i+1}}(az)\|_2\leq \frac{\varepsilon}{ 2(\kappa^2+1)^4} \text{ for all }i=\overline{1,4}.
\end{equation} 
Above, we have denoted by $\kappa \coloneqq  \max^4_{i=1} |E_i|$. Also in these formulae we make the convention that indices $5=1$.

Using Kaplansky density theorem there is $b\in \mathcal M_1$ supported on a finite subset $K=K^{-1}\subset G_1$ satisfying $\|b\|_\infty \leq 1$  and $\|az-b\|_2\leq \frac{\varepsilon}{ 2(\kappa^2+1)^4} $. Letting $b_i= b-\mathscr P_{ E_iH_1 E_{i+1}} (b)$ we see that \begin{equation}\label{sup+norm}\begin{split} &{\rm supp}(b_i)=:K_i \subseteq K \setminus E_iH_1 E_{i+1} \\&\|b_i\|_\infty\leq \kappa^2+1.
\end{split}
\end{equation} Moreover, using \eqref{nonint4} we also have 
\begin{equation}\label{approx1}\begin{split}
    \|az-b_i\|_2\leq \frac{ \varepsilon}{ (\kappa^2+1)^4}. 
\end{split}\end{equation}

Using the prior inequalities \eqref{sup+norm}, \eqref{approx1}, and \eqref{strongnorm proj} we see that

\begin{equation}\label{traceineq2}
 \begin{split}
    &\tau((az) \mathscr P_{E_1 H_1E_1}(u_g z) (b z) \mathscr P_{E_2 H_1E_2}(u_{g^{-1}} z)(a^* z) \mathscr P_{E_3H_1E_3}(u_gz) (b^*z)\mathscr P_{E_4H_1E_4}( u_{g^{-1}}z))\\
    &\leq \tau(b_1 \mathscr P_{E_1 H_1E_1}(u_g z) (a z) \mathscr P_{E_2 H_1E_2}(u_{g^{-1}} z)(a^* z) \mathscr P_{E_3H_1E_3}(u_gz) (a^*z)\mathscr P_{E_4H_1E_4}( u_{g^{-1}}z))+\\
    &\quad +\|az-b_1\|_2 \| \mathscr P_{E_1H_1E_1}(u_gz)\|_\infty\cdots  \| \mathscr P_{E_4H_1E_4}(u_gz)\|_\infty \\
     &\overset{\eqref{strongnorm proj}}{\leq} \tau(b_1 \mathscr P_{E_1 H_1E_1}(u_g z) (a z) \mathscr P_{E_2 H_1E_2}(u_{g^{-1}} z)(a^* z) \mathscr P_{E_3H_1E_3}(u_gz) (a^*z)\mathscr P_{E_4H_1E_4}( u_{g^{-1}}z))+\\
    &\quad +\|az-b_1\|_2 |E_1|^2\cdots |E_4|^2\\
     &\overset{\eqref{approx1}}{\leq} \tau(b_1 \mathscr P_{E_1 H_1E_1}(u_g z) (a z) \mathscr P_{E_2 H_1E_2}(u_{g^{-1}} z)(a^* z) \mathscr P_{E_3H_1E_3}(u_gz) (a^*z)\mathscr P_{E_4H_1E_4}( u_{g^{-1}}z))+\varepsilon\\
    &\leq \tau(b_1 \mathscr P_{E_1 H_1E_1}(u_g z) b_2 \mathscr P_{E_2 H_1E_2}(u_{g^{-1}} z)(a^* z) \mathscr P_{E_3H_1E_3}(u_gz) (a^*z)\mathscr P_{E_4H_1E_4}( u_{g^{-1}}z))+\varepsilon +\\
    &\quad \|az-b_2\|_2\|b_1\|_\infty \|\mathscr P_{E_2 H_1 E_2}(u_{g}z)\|_\infty \cdots \|\mathscr P_{E_4 H_1E_4}(u_gz)\|_\infty\\
&\overset{\eqref{sup+norm},\eqref{strongnorm proj}}{\leq} \tau(b_1 \mathscr P_{E_1 H_1E_1}(u_g z) b_2 \mathscr P_{E_2 H_1E_2}(u_{g^{-1}} z)(a^* z) \mathscr P_{E_3H_1E_3}(u_gz) (a^*z)\mathscr P_{E_4H_1E_4}( u_{g^{-1}}z))+\varepsilon +\\
    &\quad \|az-b_2\|_2 (\kappa^2+1) |E_2|^2 |E_3|^2 |E_4|^2\\
    &\overset{\eqref{approx1}}{\leq} \tau(b_1 \mathscr P_{E_1 H_1E_1}(u_g z) b_2 \mathscr P_{E_2 H_1E_2}(u_{g^{-1}} z)(a^* z) \mathscr P_{E_3H_1E_3}(u_gz) (a^*z)\mathscr P_{E_4H_1E_4}( u_{g^{-1}}z))+2\varepsilon \\
& \ldots\\
&\leq \tau(b_1 \mathscr P_{E_1 H_1E_1}(u_g z) b_2 \mathscr P_{E_2 H_1E_2}(u_{g^{-1}} z)b_1^* \mathscr P_{E_3H_1E_3}(u_gz) b_2^*\mathscr P_{E_4H_1E_4}( u_{g^{-1}}z))+4\varepsilon. 
    \end{split}   
\end{equation} 
Combining inequalities \eqref{traceineq1} and \eqref{traceineq2} imply that for all $g\in H_2$ we have that 

\begin{equation}\label{traceineq}
  \tau(z)-8\varepsilon \leq \tau(b_1 \mathscr P_{E_1 H_1E_1}(u_g z) b_2 \mathscr P_{E_2 H_1E_2}(u_{g^{-1}} z)b_1^* \mathscr P_{E_3H_1E_3}(u_gz) b_2^*\mathscr P_{E_4H_1E_4}( u_{g^{-1}}z)).  
\end{equation}

By construction, we have $F_i\coloneqq E_iK_i E_i\subset G_1\setminus H_1$. If we let $F=\bigcup^4_{i=1} F_i\subset G_1\setminus H_1$ and use Corollary \ref{nontrivial} for the finite subset $F$ and $\ell=8$ one can find a finite subset $R=R^{-1}\subset H_1$ such that \begin{equation}\label{noidentity}1\notin F (H\setminus R) F (H\setminus R) F (H\setminus R) F (H \setminus R). \end{equation} 

Finally, using \eqref{traceineq} in combination with  \eqref{strongnorm proj}, \eqref{approx1} and \eqref{sup+norm}, we see that for all $g\in H_2$ we have  
\begin{equation*}
    \begin{split}
     & \tau(z)-8\varepsilon \leq \tau(b_1 \mathscr P_{E_1 H_1E_1}(u_g z) b_2 \mathscr P_{E_2 H_1E_2}(u_{g^{-1}} z)b_1^* \mathscr P_{E_3H_1E_3}(u_gz) b_2^*\mathscr P_{E_4H_1E_4}( u_{g^{-1}}z)) \\
     & \leq \tau( b_1 \mathscr P_{E_1(H_1\setminus R)E_1}(u_gz)b_2 \mathscr P_{E_2 H_1E_2}(u_{g^{-1}} z)b_1^* \mathscr P_{E_3H_1E_3}(u_gz) b_2^*\mathscr P_{E_4H_1E_4}( u_{g^{-1}}z)+\\
     & \quad + \|\mathscr P_{E_1 RE_1}(u_g z)\|_2 \|b_1\|_\infty \|b_2\|_\infty\|b_3\|_\infty \|b_4\|_\infty \|\mathscr P_{E_2H_1E_2}(u_{g^{-1}}z)\|_\infty \|\mathscr P_{E_3H_1E_3}(u_gz)\|_\infty \|\mathscr P_{E_4 H_1 E_4}(u_{g^{-1}}z)\|_\infty\\
&\overset{\eqref{sup+norm},\eqref{strongnorm proj}}{\leq} \tau( b_1 \mathscr P_{E_1(H_1\setminus R)E_1}(u_gz)b_2 \mathscr P_{E_2 H_1E_2}(u_{g^{-1}} z)b_1^* \mathscr P_{E_3H_1E_3}(u_gz) b_2^*\mathscr P_{E_4H_1E_4}( u_{g^{-1}}z)+\\
     & \quad + \|\mathscr P_{E_1 RE_1}(u_g z)\|_2 (\kappa^2+1)^4 \kappa^6\\
     & \leq \tau( b_1 \mathscr P_{E_1(H_1\setminus R)E_1}(u_gz)b_2 \mathscr P_{E_2 (H_1\setminus R)E_2}(u_{g^{-1}} z)b_1^* \mathscr P_{E_3H_1E_3}(u_gz) b_2^*\mathscr P_{E_4H_1E_4}( u_{g^{-1}}z)+\|\mathscr P_{E_1 RE_1}(u_g z)\|_2 (\kappa^2+1)^4 \kappa^6\\
     &\quad +   \|\mathscr P_{E_2 H_1E_2}(u_{g^{-1}}z)\|_2  \|b_1\|_\infty \|b_2\|_\infty\|b_3\|_\infty \|b_4\|_\infty \|\mathscr P_{E_1 (H_1\setminus R)E_1}(u_gz)\|_\infty \|\mathscr P_{E_3H_1E_3}(u_gz)\|_\infty \|\mathscr P_{E_4 H_1 E_4}(u_{g^{-1}}z)\|_\infty\\ 
&\overset{\eqref{sup+norm},\eqref{strongnorm proj}}{\leq} \tau( b_1 \mathscr P_{E_1(H_1\setminus R)E_1}(u_gz)b_2 \mathscr P_{E_2 (H_1\setminus R)E_2}(u_{g^{-1}} z)b_1^* \mathscr P_{E_3H_1E_3}(u_gz) b_2^*\mathscr P_{E_4H_1E_4}( u_{g^{-1}}z)+\\
&\quad \|\mathscr P_{E_1 RE_1}(u_g z)\|_2 (\kappa^2+1)^4 \kappa^6 + \|\mathscr P_{E_2RE_2}(u_{g^{-1}}z)\|_2
(1+|R|)(\kappa^2+1)^4 \kappa^6\\
& \ldots \\
&\leq  \tau( b_1 \mathscr P_{E_1(H_1\setminus R)E_1}(u_gz)b_2 \mathscr P_{E_2 (H_1\setminus R)E_2}(u_{g^{-1}} z)b_1^* \mathscr P_{E_3(H_1\setminus R)E_3}(u_gz) b_2^*\mathscr P_{E_4(H_1\setminus R) E_4}( u_{g^{-1}}z)+\\
&\quad +(\kappa^2+1)^4 \kappa^4 (\sum^4_{i=1}\|\mathscr P_{E_i RE_i}(u_g z)\|_2 (1+|R|)^{i-1}).
\end{split}
\end{equation*}
 Now observe that the element inside the trace above is supported on $K_1 E_1 (H_1\setminus R) F (H_1\setminus R) F (H_1\setminus R) F (H_1\setminus R)E_4$. However, by \eqref{noidentity} this set cannot contain the identity of $G_1$. Therefore the first term in the sum above vanishes. In conclusion, for every  $g\in H_2$ we have  \begin{equation}\label{traceineq3}
  \tau(z)-8\varepsilon \leq (\kappa^2+1)^4 \kappa^4\left (\sum^4_{i=1}\|\mathscr P_{E_i RE_i}(u_g z)\|_2 (1+|R|)^{i-1})\right ).   
 \end{equation} 
Finally, since $L(H_2)$ has no amenable direct summand and $\mathcal A$ is abelian then $L(H_2)z\nprec \mathcal A$ and hence by Popa's intertwining technique there is a sequence $(g_n)_n \subset H_2$ such that $\|\mathscr P_{E_iRE_i}(u_{g_n} z)\|_2\rightarrow 0$ as $n\rightarrow \infty$ for all $i=\overline{1,4}$. Letting $g=g_n$ in \eqref{traceineq3} and taking the limit over $n$ we get $\tau(z)\leq 8\varepsilon$ which is a contradiction for $\varepsilon> 0$ sufficiently small. Therefore, we conclude that $\mathcal B z\prec_{\mathcal M_1} \mathcal A\rtimes H_2$. $\hfill\blacksquare$ 
\vskip 0.06in
It remains to show that $\mathcal Bz\prec_{\mathcal M_1} \mathcal A$. Assume for contradiction that this does not hold. Then, by Claim \ref{intcartan} and Theorem \ref{intmalnormal}, the von Neumann algebra $\mathcal N_{z\mathcal M_2 z}(\mathcal Bz)''$ intertwines into $\mathcal A\rtimes H_1$. This implies that $\mathcal M_1 \prec_{\mathcal M_1} \mathcal A \rtimes H_1$, which in turn entails that $[G_1:H_1]<\infty$, a contradiction. Thus $\mathcal B \prec_{\mathcal M_1} \mathcal A$, which further implies that $\mathcal B\prec_{\mathcal M_1} \mathcal A p$ and hence $\mathcal B \prec_{\mathcal M_2} \mathcal A p$. Since $\mathcal B$ and $\mathcal A p$ are Cartan subalgebras in the II$_1$ factor $\mathcal M_2$, Popa's unitary conjugacy criterion \cite[Theorems A.1 and 6.2]{Po01} yields a unitary $u\in \mathcal M_2$ such that $u\mathcal B u^*=\mathcal A p$, and the desired conclusion follows.\end{proof}

\subsection{An ME-rigidity result for relatively biexact groups}

We now establish an ME-rigidity result for relatively biexact groups (Theorem \ref{thm:MEfactorrigid}). 

The argument follows closely the methods developed by Sako in \cite[Sections 7--8]{Sak09}, in particular the maximal partial embedding argument of \cite[Lemma 45]{Sak09} and the factorization argument for amalgamated free products in \cite[Proposition 49 and Theorem 50]{Sak09}. For the reader's convenience, we include the relevant arguments, adapted to the present setting.

\vskip 0.05in

We start by briefly recalling the the measurable embedding terminology and notation that will be used
throughout this subsection; see \cite[Definitions 2, 9 and 12 and
Subsection~3.3]{Sak09}.
\vskip 0.05in
Let $G$ and $\Gamma$ be countable groups. A \emph{measurable embedding}
of $G$ into $\Gamma$ is a standard measure space $(\Sigma,\nu)$
equipped with a measure-preserving $G\times\Gamma$-action for which there
exist a $\Gamma$-fundamental domain $X\subset\Sigma$ of finite measure
and a $G$-fundamental domain $Y\subset\Sigma$, not necessarily of finite
measure. We denote the existence of such an embedding by
\[
G\preceq_{\mathrm{ME}}\Gamma.
\]
The measurable embedding $\Sigma$ is called \emph{ergodic} if the
$G\times\Gamma$-action on $\Sigma$ is ergodic. If $Y$ also has finite
measure, then $\Sigma$ is an ME coupling of $G$ and $\Gamma$;
see \cite[Definition~9]{Sak09}.

More generally, let $H<G$ and $\Lambda<\Gamma$ be subgroups, and let
$\Sigma$ be a measurable embedding of $G$ into $\Gamma$. We say that
$H$ \emph{measurably embeds} into $\Lambda$ in $\Sigma$, and write
\[
H\preceq_{\Sigma}\Lambda,
\]
if there exists a non-null measurable subset $\Omega\subset\Sigma$ which
is invariant under the $H\times\Lambda$-action and for which a
$\Lambda$-fundamental domain has finite measure. Following \cite[Definition~2]{Sak09}, such a set $\Omega$
is called a \emph{partial embedding} of $H$ into $\Lambda$. 

Choose a $\Gamma$-fundamental domain $X\subset\Sigma$. The
\emph{$\Gamma$-support} of $H\preceq_{\Sigma}\Lambda$, denoted by
\[
\operatorname{supp}_{X}^{\Gamma}
   (H\preceq_{\Sigma}\Lambda)\in L^\infty(X),
\]
is the projection corresponding, under the identification $(L^\infty(\Sigma))^\Gamma\simeq L^\infty(X)$, to
\[
\bigvee
\left\{
\gamma\chi_{\Omega}:
\gamma\in\Gamma,\;
\Omega\subset\Sigma
\text{ is a partial embedding of }H\text{ into }\Lambda
\right\}
\in (L^\infty(\Sigma))^\Gamma.
\]
Analogously, after choosing a $G$-fundamental domain $Y\subset\Sigma$,
one defines the $G$-support to be  
$\operatorname{supp}_{Y}^{G}(H\preceq_{\Sigma}\Lambda)$;
see \cite[Definition~12]{Sak09}.
\vskip 0.05in
For further use we also recall the function-valued measure. For a subgroup
$\Lambda<\Gamma$, let $X_\Lambda$ be a $\Lambda$-fundamental domain
for $\Sigma$, and let $\operatorname{Tr}_{\Lambda}$ denote integration
over $X_\Lambda$. If $X$ is a fixed $\Gamma$-fundamental domain, the
natural inclusion
\[
\iota:L^\infty(X)\longrightarrow (L^\infty(\Sigma))^\Lambda,
\qquad
\iota(f)(\gamma x)=f(x)
\quad (x\in X,\ \gamma\in\Gamma),
\]
induces the function-valued measure
\[
E_X^\Lambda:
L^1\bigl((L^\infty(\Sigma))^\Lambda,\operatorname{Tr}_{\Lambda}\bigr)
\longrightarrow L^1(X),
\]
which extends completely additively to the extended positive cones. If
$\{s_\iota\}_{\iota\in I}$ is a set of representatives for the right
cosets $\Lambda\backslash\Gamma$, then
\[
E_X^\Lambda(\varphi)(x)
=
\sum_{\iota\in I}\varphi(s_\iota x),
\qquad
\varphi\in (L^\infty(\Sigma))^\Lambda_+.
\]
Finally, as in \cite[Subsection~3.3]{Sak09}, for a $\Lambda$-invariant measurable subset $\Omega\subset\Sigma$, we
write
\[
E_X^\Lambda(\Omega):=E_X^\Lambda(\chi_\Omega).
\]
\vskip 0.05in
With these preparations at hand we first establish an analogue of the maximal partial embedding argument in \cite[Lemma 45]{Sak09}, with almost malnormality replacing the wreath-product structure used there.

\begin{lemma}\label{lem:maximal-malnormal}
Let $\Gamma$ be a group and let $\Gamma_0<\Gamma$ be an almost malnormal subgroup. Let $\Sigma$ be a measurable embedding of a countable group
$G$ into $\Gamma$. Let $H<G$ be a non-amenable group such that $H\preceq_\Sigma \Gamma_0$.

\noindent Then there exists a maximal partial embedding $\Omega\subset\Sigma$ of $H$ into $\Gamma_0$. Namely, if $\Omega'$ is any other partial embedding of
$H$ into $\Gamma_0$, then $\Omega'\setminus\Omega$ is null.

\noindent Moreover,
\begin{equation}\operatorname{supp}_X^\Gamma(H\preceq_\Sigma\Gamma_0)
=
E_X^{\Gamma_0}(\Omega).
\end{equation}
\end{lemma}

\begin{proof} 
    Let $\{s_\iota\}_{\iota\in I}$ be representatives of the right cosets $\Gamma_0\backslash\Gamma$. Let $\Omega\subset\Sigma$ be an arbitrary partial embedding of $H$ into $\Gamma_0$, and let $X$ be a $\Gamma$-fundamental domain.

Then we have  $\Omega = \bigsqcup_{\iota\in I} \Gamma_0 s_\iota X_\iota$, for soem measurable subsets $X_\iota\subset X$. Hence
\begin{equation*}
    E_X^{\Gamma_0}(\Omega) = \sum_{\iota\in I}\chi_{X_\iota},
\end{equation*}
which is integrable.
\begin{claim}\label{proj}
    $E_X^{\Gamma_0}(\Omega)$ is a projection.
\end{claim} 
\noindent\emph{Proof of Claim \ref{proj}}
Assume by contradiction the essential range of $E_X^{\Gamma_0}(\Omega)$ is not contained
in $\{0,1\}$. Then there exist a non-null measurable subset $W\subset X$ and
distinct representatives
$s_1,\dots,s_k, k\ge2$, such that
\[
\Omega\cap\Gamma W
=
\bigsqcup_{i=1}^k
\Gamma_0s_iW.
\]

Replacing $X$ with $s_1W\sqcup(X\setminus W)$ and the representatives $\{s_i\}$ with
$\{s_is_1^{-1}\}$, we may assume $s_1=1$. The measurable subset $s_2\Omega\cap\Omega$ is $H$-invariant and satisfies
\[
s_2\Omega\cap\Omega
=
\bigcup_{i,j}
(s_2\Gamma_0s_i\cap\Gamma_0s_j)W.
\]
Since $\Gamma_0$ is almost malnormal and $s_2\notin\Gamma_0$, each set $s_2\Gamma_0s_i\cap\Gamma_0s_j$ is either empty or finite. Hence $S\coloneqq 
\bigcup_{i,j}
(s_2\Gamma_0s_i\cap\Gamma_0s_j)$
is a finite subset of $\Gamma$.

Therefore $s_2\Omega\cap\Omega=SW$,
up to null sets. Since we can view $S$ as a finite union of right cosets of the trivial
subgroup,
\[
E_X(s_2\Omega\cap\Omega)|_W
=
|S|\,1_W.
\]
By \cite[Lemma 15]{Sak09}, we obtain $H\preceq_\Sigma\{1\}$, contradicting the non-amenability of $H$.

Hence the essential range of $E_X^{\Gamma_0}(\Omega)$ is contained in
$\{0,1\}$.$\hfill\blacksquare$

\vskip 0.06in

For the rest of the proof, we recycle the same argument as in \cite[Lemma 45]{Sak09}. When $\Omega,\Omega'$ are partial embeddings of $H$ into $\Gamma_0$, the union
$\Omega\cup\Omega'$ is also a partial embedding of $H$ into $\Gamma_0$. By the
above, $E_X^{\Gamma_0}(\Omega\cup\Omega')$ is a projection.

There exists an increasing sequence $\{\Omega_n\}$ of partial embeddings of
$H$ into $\Gamma_0$ such that
\[
\bigvee_nE_X^{\Gamma_0}(\Omega_n)
=
\operatorname{supp}_X^\Gamma(H\preceq_\Sigma\Gamma_0).
\]

Let $\Omega=\bigcup_n\Omega_n$. Applying $E_X^{\Gamma_0}$ gives $E_X^{\Gamma_0}(\Omega)
=
\sup_nE_X^{\Gamma_0}(\Omega_n)
=
\operatorname{supp}_X^\Gamma(H\preceq_\Sigma\Gamma_0)$.
It follows that $\Omega$ is again a partial embedding.

Let $\Omega'$ be another partial embedding. Then
$E_X^{\Gamma_0}(\Omega)
\le
E_X^{\Gamma_0}(\Omega\cup\Omega')
\le
\operatorname{supp}_X^\Gamma(H\preceq_\Sigma\Gamma_0)
=
E_X^{\Gamma_0}(\Omega)$.
Since $E_X^{\Gamma_0}$ is faithful, $\chi_{\Omega\cup\Omega'}
=
\chi_\Omega$,
so $\Omega'\setminus\Omega$ is null.

Altogether, these prove the maximality of $\Omega$ and the formula for its
$\Gamma$-support.
\end{proof}

\vskip 0.05in
We next prove the analogue of \cite[Proposition~46]{Sak09}
in our setting.

\begin{prop}\label{prop:malnormal-maximal}
Let $\Gamma$ be a countable group which is bi-exact relative to an almost malnormal subgroup
$\Gamma_0<\Gamma$. Let $G_1\times H<G$
be a direct product subgroup of a countable group $G$, where both
$G_1$ and $H$ are nonamenable. Suppose that $\Sigma$ is an ergodic measurable embedding of $G$ into
$\Gamma$.

\noindent Then one can find a maximal partial embedding
$\Omega\subset\Sigma$
of $G_1\times H$ into $\Gamma_0$. Moreover,
$E_X^{\Gamma_0}(\Omega)=1_X$.
\end{prop}

\begin{proof}
Since $\Gamma$ is bi-exact relative to $\{\Gamma_0\}$, \cite[Theorem 25]{Sak09} implies
that $H\preceq_\Sigma\Gamma_0.
$
Furthermore, the $\Gamma$-support of the embedding is $1_X$. Let
$\Omega$ be the largest embedding of $H$ into $\Gamma_0$
(Lemma~\ref{lem:maximal-malnormal}).

Since every element of $G_1$ commutes with every element of $H$, for
every $g\in G_1$ the measurable subsets
$g\Omega,g^{-1}\Omega$
are again partial embeddings of $H$ into $\Gamma_0$.

By maximality of $\Omega$, both
$g\Omega\subset\Omega,
g^{-1}\Omega\subset\Omega$,
modulo null sets. Hence the symmetric difference between $g\Omega$ and
$\Omega$ is null. Replacing $\Omega$ by an equivalent measurable subset,
we may therefore assume that $\Omega$ is $G_1$-invariant. Since $\Omega$
is already $H$-invariant by construction, it follows that $\Omega$ is
$(G_1\times H)\times\Gamma_0$-invariant. Consequently, $\Omega$ is a measurable embedding of
$G_1\times H$ into $\Gamma_0$.

The embedding $\Omega$ is maximal as an embedding of
$G_1\times H$, since every partial embedding of $G_1\times H$ into
$\Gamma_0$ is, a fortiori, a partial embedding of $H$ into
$\Gamma_0$, and $\Omega$ is maximal among such embeddings.

Finally, $E_X^{\Gamma_0}(\Omega)
=
\operatorname{supp}_X^\Gamma(H\preceq_\Sigma\Gamma_0)
=
1_X$,
as claimed.
\end{proof}

We are now ready to prove the ME-rigidity result needed below. Its proof is an adaptation of the factor-matching argument in the proof of \cite[Theorem~50]{Sak09}.

\begin{thm}\label{thm:MEfactorrigid} For each $i=1,2$, let $H_i<G_i$ be groups satisfying the following assumptions:

\begin{enumerate}
\item $H_i=H_i^1\times H_i^2$, where $H_i^1$ and $H_i^2$ are ICC and nonamenable;
\item $G_i$ is bi-exact relative to $\{H_i\}$;
\item $H_i$ is almost malnormal in $G_i$.
\end{enumerate}

If $G_1\simeq_{\rm ME}G_2$,
then $H_1\simeq_{\rm ME}H_2$.
\end{thm}

\begin{proof}
Let $\Sigma$ be an ergodic ME coupling of $G_1$ and $G_2$. By Proposition~\ref{prop:malnormal-maximal}, there exists a maximal partial embedding
$\Omega\subset\Sigma$
of $H_1$ into $H_2$. Likewise, exchanging the roles of the two groups, there exists
a maximal partial embedding $\Xi\subset\Sigma$
of $H_2$ into $H_1$.

We first show that
\[
\chi(\Omega)\leq \chi(\Xi).
\]

Suppose otherwise. Then there exists
$g\in G_1\setminus H_1$
such that $\Omega\cap g\Xi$
is non-null.

Since $H_1$ is almost malnormal, $H_1\cap gH_1g^{-1}$ is finite. Let
$\{h_\iota\}_{\iota\in I}$
be representatives for the left cosets
$H_1/(H_1\cap gH_1g^{-1})$.
As $H_1$ is infinite, the index set $I$ is infinite.

Consider $Y_\Xi\subset \Xi$ an $H_1$-fundamental domain, so that $\Xi=H_1 Y_\Xi$. Next we argue that the measurable subsets $\{h_\iota g\Xi\}_{\iota\in I}$
are pairwise disjoint modulo null sets. Towards this, assume by contradiction that $h_\iota g\Xi\cap h_\kappa g\Xi$
is non-null for some $\iota\neq\kappa$. Since $\Xi= H_1 Y_\Xi$, this further implies existence of $a,b\in H_1$ so that $h_\iota ga Y_\Xi\cap h_\kappa g b Y_\Xi$ is non-null. Since $Y_\Xi$ is  $H_1$-fundamental domain then $a^{-1}g^{-1}h^{-1}_\iota h_\kappa gb \in H_1$ and hence $g^{-1}h^{-1}_\iota h_\kappa g \in H_1$. Therefore $h_\iota^{-1}h_\kappa \in gH_1 g^{-1}\cap H_1$, which contradicts the choice of representatives.  Hence $\{h_\iota g\Xi\}_{\iota\in I}$ are
pairwise disjoint modulo null sets.

Since $\Omega$ is $H_1$-invariant, $h_\iota(\Omega\cap g\Xi)
=
\Omega\cap h_\iota g\Xi$. Therefore,
\begin{equation}\label{pos}
\operatorname{Tr}_{H_2}
(\Omega\cap h_\iota g\Xi)
=
\operatorname{Tr}_{H_2}
(\Omega\cap g\Xi)
>0 \text{ for every }\iota\in I.
\end{equation}
Using the pairwise disjointness of the translates,
\[
\begin{aligned}
0
&<
\sum_{\iota\in I}
\operatorname{Tr}_{H_2}
(\Omega\cap h_\iota g\Xi) =
\operatorname{Tr}_{H_2}
\left(
\Omega\cap
\bigcup_{\iota\in I}
h_\iota g\Xi
\right)\le
\operatorname{Tr}_{H_2}(\Omega)
<\infty.
\end{aligned}
\]
By \eqref{pos}, this further implies that $|I| {\rm Tr}_{H_2}(\Omega \cap g \Xi)<\infty$, contradicting that $I$ is infinite.

Hence $\Omega\cap g\Xi$ is null for every $g\in G_1\setminus H_1$ and  consequently, $\chi(\Omega)\leq\chi(\Xi)$.

Interchanging the roles of $H_1 <G_1$ and $H_2<G_2$ we also get 
$\chi(\Xi)\leq\chi(\Omega)$. Therefore, $\chi(\Xi)=\chi(\Omega)$ which implies that $\Xi=\Omega$ modulo null sets. Finally, this yields $H_1\sim_{\rm ME} H_2$ via coupling $\Xi$.\end{proof}

\subsection{A continuum family of property (T) relative hyperbolic groups}

Let $\mathscr I=(1,\frac{9}{8})$. For every $c\in \mathscr I$, Gaboriau constructed in \cite{Ga00} a group $J_c$ as the fundamental
group of an infinite graph of finite groups
\[
G_1
\;\xleftrightarrow{\;K_1\;}
G_2
\;\xleftrightarrow{\;K_2\;}
G_3
\;\xleftrightarrow{\;K_3\;}
\cdots
\]
where

\begin{enumerate}
\item each vertex group $G_n$ is finite;
\item each edge group $K_n$ contains an index-$2$ subgroup $H_n$;
\item for $n\ge2$, the group $G_n$ is generated by $K_{n-1}$ and $H_n$;
\item the inclusions $H_1<H_2<\cdots$ have locally finite (hence amenable) union $H_\omega=\bigcup_{n\ge1}H_n$.
\end{enumerate}

The finite groups are chosen so that repeated applications of Gaboriau's
cost formula for amalgams over amenable subgroups yield $C(J_c)=c$.

Since $H_\omega$ is not finitely generated, $J_c$ need not be finitely generated. Thus Gaboriau embeds $H_\omega$
into a finitely generated amenable group $H$ and defines $\Gamma_c=J_c*_{H_\omega}H$. Using the cost formula for amalgamated free products one obtains $C(\Gamma_c)=C(J_c)=c$.

Hence every $c\ge1$ is realized as the cost of a finitely generated
fixed-price group. \cite[Proposition~VI.16]{Ga00} Since cost is an orbit-equivalence invariant for
fixed-price groups, the groups $\Gamma_c$ corresponding to distinct
values of $c$ are pairwise non-orbit equivalent. \cite[Proposition~VI.16]{Ga00}. Now as in \cite[Proposition 5.1]{ITD} consider the free product $A_c= \Gamma_c\ast (\mathbb F_2\times \mathbb F_2)$. 

\begin{prop}\label{exactgroups}
  For every $c_1,\ldots,c_n\in \sI$, the product group   $A_{c_1}\times \cdots \times A_{c_n}$ is exact. 
\end{prop}

\begin{proof} First fix $c>1$.
    Since $J_c$ is strongly treeable and has fixed price $c>1$, every free ergodic probability measure preserving action of $J_c$ gives rise to an ergodic treeable equivalence relation of finite cost $c$. By Gaboriau's realization theorem for treeable equivalence relations, every ergodic treeable equivalence relation of finite cost is stably orbit equivalent to a free probability measure preserving action of a nonabelian free group. Since stable orbit equivalence of free ergodic actions is equivalent to measure equivalence of the acting groups (see, e.g., \cite{Fu99}), it follows that $J_c$ is measure equivalent to a nonabelian free group. Since exactness is preserved under measure equivalence by \cite[Theorem 0.1(6)]{DL15} it follows that $J_c$ is exact. Finally, since exactness is preserved under amalgamated free product over amenable subgroup we concludde that $\Gamma_c
$ is also exact and hence $A_c
$ is also exact. Since exactness passed to direct products the conclusion follows.
\end{proof}

Next for any unordered pair $\imath=\{c,d\}\subset \sI$ we consider the product group $H_\imath= A_c\times A_d$. 

\begin{thm}\label{nonMEperipheral}
   For every $\imath\neq \jmath$ we have that $H_\imath \not\simeq_{\rm ME} H_\jmath$. 
\end{thm}

\begin{proof}
Fix $\imath=\{c,d\}$ and $\jmath =\{e,f\}$ and let $H_\imath = A_c\times A_d$ and $H_\jmath= A_e\times A_f$. Assume that $H_\imath\simeq_{\rm ME}H_\jmath$. Since the factors of these groups are free products then using \cite[Theorem A]{Dr23} and after permuting the indices if necessarily, we get that $A_c\simeq_{\rm ME} A_e$ and $A_d\simeq_{\rm ME} A_f$. From \cite[Proposition 5.1]{ITD} this further implies that $c=e$ and $d=f$ and hence $\imath=\jmath$.  
\end{proof}

\vskip 0.05in

\noindent\emph{Proof of Theorem \ref{mainresult:nonW^*e}}. For each $\imath=\{c,d\}$ with $c,d\in \sI$ consider the group $H_\imath= A_c \times A_d$ as above. Since $A_c$ and $A_d$ are finitely generated so is $H_\imath$.  Using \cite{AMO06}, one can find a supragroup $H_\imath<G_\imath$ that has property (T) and it is hyperbolic relative to $H_\imath$.  Since $H_\imath$ is exact by Proposition \ref{exactgroups} it follows from Theorem \ref{thm: main} that $G_\imath$ is biexact relative to $\{H_\imath\}$. Moreover, relative hyperboliciy implies that $H_\imath$ is an almost malnormal in $G_\imath$. For $\imath,\jmath$ let $G_\imath \curvearrowright Z_\imath$ $G_\jmath\curvearrowright Z_\jmath$ be  a pair of arbitrary free, ergodic pmp actions on standard probability spaces. Denote by $\sM_\imath= L^\infty(Z_\imath)\rtimes G_\imath$ and $\sM_\jmath=L^\infty(Z_\jmath)\rtimes G_\jmath$ the corresponding II$_1$ factors and assume that there is a projection $p\in M_\imath$ such that $p\sM_\imath p =M_\jmath$. By factoriality we can assume that $p\in L^\infty(Z_\imath)$. Since $H_\imath$ are products of nonamenable groups Theorem \ref{uniquecartan} implies that there is a unitary $u\in \sM_\jmath$ such that $uL^\infty(Z_\imath)pu^*= L^\infty(Z_\jmath)$. Then   \cite{Si55} further implies the actions $G_\imath\curvearrowright Z_\imath$ and $G_\jmath\curvearrowright Z_\jmath$ are stably orbit equivalent. Thus $G_\imath \simeq_{\rm ME} G_\jmath$ and by Theorem \ref{thm:MEfactorrigid} it follows that $H_\imath\simeq_{\rm ME } H_\jmath$ and hence by Theorem \ref{nonMEperipheral} we get $\imath=\jmath$.

\end{document}